\documentclass[reqno,11pt]{amsart}
\usepackage{color} 
\usepackage[left=1 in, right=1 in,top=1 in, bottom=1 in]{geometry}
\usepackage{amsfonts,amsmath,amsthm,amssymb,mathrsfs}
\newtheorem{theorem}{Theorem}[section]
\newtheorem{definition}[theorem]{Definition}
\newtheorem{lemma}[theorem]{Lemma}
\newtheorem{proposition}[theorem]{Proposition}
\newtheorem{remark}[theorem]{Remark}

\numberwithin{equation}{section}
\usepackage[pdfstartview=FitH,colorlinks=true,backref=page]{hyperref}
\hypersetup{urlcolor=red, linkcolor=blue,citecolor=red, CJKbookmarks=true}
\allowdisplaybreaks
\makeatletter
\newcommand*{\rom}[1]{\expandafter\@slowromancap\romannumeral #1@}
\newcommand{\rcong}{\stackrel{r}{\rightharpoonup}}

\newcommand{\bd}{\mathrm{d}}
\newcommand{\mm}{\mathrm{M}}
\newcommand{\grad}{\nabla_{x}}
\newcommand{\vdot}{v\!\cdot\!\grad}
\newcommand{\Q}{\mathcal{Q}}
\newcommand{\xmus}{ \dd \mu \bd \sigma_x \bd s }
\newcommand{\uu}{\mathrm{u}}

\makeatother
\usepackage{tikz}

\def \dd{\mathrm{d}}
\def\la {\langle}
\def\ra {\rangle}
\newcommand{\weak}{\rightharpoonup}
\def\B {\mathcal{B}}

\def\eps {{\epsilon}}
\title[Existence and NS limits]{The Boltzmann equation with incoming boundary condition: global solutions and Navier-Stokes limit}

\author[N. Jiang]{Ning Jiang}
\address[Ning Jiang]{\newline School of Mathematics and Statistics, Wuhan University, Wuhan, 430072, P. R. China}
\email{njiang@whu.edu.cn}

\author[X. Zhang]{Xu Zhang}
\address[Xu Zhang]{\newline School of Mathematics and Statistics, Wuhan University, Wuhan, 430072, P. R. China}
\email{xuzhang889@whu.edu.cn}
\begin{document}
	\begin{abstract}
We consider the Boltzmann equations with cutoff collision kernels in bounded domains. For the initial data with finite physical bounds, we prove the existence of global-in-time renormalized solutions in the sense of DiPerna-Lions endowed with incoming boundary condition. Moreover, we justify the limit as the Knudsen number $\eps\rightarrow 0$ to Leray solutions of the incompressible Navier-Stokes-Fourier equations with homogeneous Dirichlet conditions from renormalized solutions of the scaled Boltzmann equations when the incoming data are close to the global Maxwellian in the sense of the boundary relative entropy of order $O(\eps^3)$.
	\end{abstract}
\maketitle	

\section{Introduction}
The Boltzmann (or Maxwell-Boltzmann) equation is an integro-differentiable equation
\begin{equation}\label{BE}
\partial_t F + \vdot F = \Q(F,F)\,,
\end{equation}
which models the statistical evolution of a rarefied gas. In equation \eqref{BE}, $F(t,x,v)$ is a non-negative measurable function, which denotes the number density of the gas molecules at time $t \geq 0$, at the position $ x \in \Omega$, with velocity $v \in \mathbb{R}^3$. Here $\Omega$ is the space domain which could be the whole space $\mathbb{R}^3$, a torus $\mathbb{T}^3$, or a bounded (or unbounded) domain in $\mathbb{R}^3$ with boundary $\partial \Omega \neq \varnothing$ (in the current paper, $\Omega$ will be this case). Furthermore, $\Q(F, F)$ is the collision operator whose structure is described below.

The Boltzmann equation \eqref{BE} is given an initial data which satisfies some natural physical bounds (bounded mass, momentum, energy and entropy, etc.). More specifically,
\begin{equation}\label{IC-1}
F|_{t=0} = F_0(x,v)\quad \mbox{in}\quad\! \Omega \times \mathbb{R}^3\,,
\end{equation}
which satisfies
\begin{equation}\label{IC-bounds}
F_0 \geq 0\quad \mbox{a.e.}\quad \mbox{and}\quad\! \iint_{  \mathbb{R}^3 \times \Omega}F_0(1+ |x|^2 + |v|^2 + |\log F_0|)\,\bd v\bd x < \infty\,.
\end{equation}

The well-posedness of the Boltzmann equation \eqref{BE} is a fundamental problem in mathematical physics. Besides many results on the smooth solutions which required the initial data $F_0$ is ``small" in some functional spaces, the first global in time solution with ``large" data, i.e. the initial data $F_0$ satisfies \eqref{IC-bounds}: only some finite physical bounds, without any smallness requirements on the size of $F_0$, was proved in DiPerna-Lions' theorem \cite{diperna-lions1989cauchy}  for $\Omega = \mathbb{R}^3$. Since in the natural functional spaces of the number density $F(t,x,v)$, say $L^1\cap L\log L$, the collision term $\Q(F,F)$ in \eqref{BE} is not even locally integrable, which makes weak solutions to the Boltzmann equation can not be defined in the usual sense.  Instead, under the Grad's angular cutoff assumption and a mild decay condition on the collision kernel which we will describe in details later, DiPerna and Lions defined the so-called renormalized solutions of \eqref{BE} and proved that a sequence of renormalized solutions which satisfy only the physically natural {\em a priori} bounds is weakly compact in $L^1$. From this stability they deduced global existence of renormalized solutions. Later on, in \cite{alex-villani2002noncutoff} Alexandre and Villani extended the results of \cite{diperna-lions1989cauchy} to the Boltzmann equations with kernels of some long-range interactions.

The Boltzmann equation in domains with boundary is more physically relevant, both in applied fields and theoretical research. For general setting of the initial-boundary value problems, see \cite{cercignani1988book} and \cite{cercignani-illuner-pul2013book-dilutegases}. Among many boundary conditions which can be imposed on the Boltzmann equations, we introduce here two simplest prototypes: Maxwell reflection and incoming boundary conditions. The former was proposed by Maxwell in \cite{maxwell1879boundaryconditions},  which  stated that the gas molecules back to the domain at the boundary come one part from the specular reflection of the molecules escaping the domain, the other part from those entering the wall, interacting with the molecules in the wall, and re-evaporating back to the domain with the thermal dynamical equilibrium state of the wall. The latter is more direct: the number density of the gas molecules back to the domain is prescribed. Most of the existing literatures focused on these two boundary conditions, or their linear combinations.

After DiPerna-Lions' work, extending their results to domains with boundary is a natural but nontrivial question. One of the main difficulty is that the functional space for renormalized solutions is very weak (say, $L\log L$), and the space for the trace of the solutions is even weaker. Thus, besides the compactness and stability of the sequence of traces are hard to establish, even the definition of the trace itself is a nontrivial issue. Most of the previous work could only establish inequality for the trace of the boundary values, see \cite{hamdache1992initial, mischler2000vlasovtrace, mischler2000vpb}. In particular, in \cite{mischler2000vpb}, an inequality for the boundary condition of the linear combination of Maxwell and incoming boundary conditions was proved.

The first complete answer in this direction is due to Mischler \cite{mischler2010asens}, who proved global renormalized solutions of many types of the kinetic equations (including Boltzmann equations) with Maxwell condition,  based on some new observations on weak-weak convergence and his previous results on the traces of kinetic equations \cite{mischler2000vlasovtrace, mischler2000vpb}. We emphasize that in \cite{mischler2010asens}, the equality for the traces was established.

In early 90's, starting from \cite{bgl1991formal, bgl1993convergence} Bardos-Golse-Levermore initialed an program (briefly, BGL program) to justify hydrodynamic limits in the framework of renormalized solutions. After many attempts to overcome the technical difficulties left in \cite{bgl1993convergence}, see \cite{bgl2000acoustic, lm2001acoustic, gl2002stokes,golse-srm-2002-l1,srm2003}, finally, the first complete rigorous Navier-Stokes-Fourier (briefly, NSF) limit was obtained by Golse and Saint-Raymond\cite{gsrm2004} for a class of bounded collision kernel. Later, they \cite{gsrm2009ns-hard} generalized their result to the cutoff hard potential kernel. Levermore and Masmoudi   \cite{lm2010soft} considered the NSF limits for  general potential kernels including soft potentials.

For bounded domains, based on Mischler's result the Boltzmann equation with Maxwell boundary, Masmoudi and Saint-Raymond \cite{masmoudi-srm2003stokesfourier} justified the Stokes limit, and the weak convergence was extended to NSF limit in \cite{jm2017layer} and \cite{saint2009book}. Furthermore, in \cite{jm2017layer}, by employing the kinetic-fluid boundary layer which damped the fast oscillating acoustic waves, the limit was enhanced to strong convergence.

All the above results, both existence of renormalized solutions and hydrodynamic limits in bounded domains, are about the Maxwell reflection boundary. The corresponding results for the incoming boundary are the main concern of the current paper. Overall, regarding to renormalized solutions, Mischler's method in \cite{mischler2010asens} can be employed here. In fact, for the incoming data case we can obtain better estimates on the traces, and as a consequence, the proof is even shorter. Furthermore, we emphasize that for the hydrodynamic limits, the incoming condition has quite different features with Maxwell condition. We explain the details as follows.

We take the acoustic limit of the Boltzmann equation as an example. From Maxwell condition, the acoustic system derived from the Boltzmann equation is endowed with the {\em impermeable} boundary condition, i.e. $\mathrm{u}\cdot\mathrm{n}=0$ (here $\mathrm{u}$ is the velocity and $\mathrm{n}$ is the outer normal vector on the boundary), see \cite{jlm2010acoustic}. However, as formally derived in \cite{agj2017acoustic}, from the Boltzmann equation with incoming condition, the acoustic system has the {\em jump} boundary condition $\rho + \theta = \alpha (\mathrm{u}\!\cdot\!\mathrm{n})$, where the constant $\alpha > 0$ is determined from Bardos-Caflisch-Nicolenako theory  \cite{bcn1986} of kinetic boundary layer equation. From the classical theory for the constant coefficients linear first order hyperbolic system (for example, see \cite{bg-s-hyperbolic}), the impermeable and jump conditions are the only two boundary conditions which make the acoustic system well-posed. From the hydrodynamic limits point of view, for the Boltzmann equation, the incoming condition is a different type with Maxwell condition. Furthermore, for the incompressible Navier-Stokes limits, the boundary condition derived from kinetic equations with incoming data has also many new features comparing with that of Maxwell condition. For more details of formal analysis, see Sone group's works in \cite{sone2002book} and \cite{sone2007}. In this sense, the Boltzmann equation with incoming data has the interests of its own to be investigated. This is the main concern and the novelty of this paper.

This paper includes two main results. The first is global existence of renormalized solutions for the Boltzmann equation with incoming data. We prove this basically following the strategy of Mischler \cite{mischler2010asens, mischler2000vlasovtrace, mischler2000vpb}. In fact, we proved that the traces of solutions for incoming data have better estimates, compared to Maxwell condition in the following sense: for the Maxwell condition, using the {\it Darroz\`es-Guiraud} information, only partial trace estimates are obtained. More specifically, only the estimate for the fluctuations to the trace average was established. For incoming data, we obtain the full estimates for traces. Moreover, we have a new discovery that for local conservation law of mass, the trace operator is commutative with the integration operator in the direction orthogonal to the boundary, see Remark \ref{remark}. This new commutative property is very useful in verifying acoustic limit (see \cite{jlm2010acoustic}), but it was unknown  in \cite{mischler2010asens} for Maxwell condition.

The second result of this paper is about the incompressible NSF limit from the Boltzmann equation with ``well-prepared" incoming boundary, i.e. the boundary data is close to the global Maxwellian M$(v)=(\sqrt{2\pi})^{-3} \exp(-\frac{|v|^2}{2})$ in the sense of the incoming boundary relative entropy with order $O(\eps^3)$ (while the initial data are close to M($v$) with order $O(\eps^2)$ respectively). In this case, we justify the convergence to NSF system with homogeneous Dirichlet condition. We emphasize that for the bounday data not close to $\mm(v)$, the uniform entropy bound is not available, which makes the incompressible fluid limit very hard to be justified. We leave this to the future work. Nevertheless, to our best acknowledgement, the NSF limit obtained in this paper is the first rigorous justification of the diffusive limit from the Boltzmann equation with incoming boundary condition.

In this rest of this Introduction, we first introduce more detailed information on the Boltzmann equation in particular the collision kernels and the boundary condition so that we can state our main results precisely.

\subsection{Boltzmann collision kernels}

In Boltzmann equation \eqref{BE}, $\Q$ is the Boltzmann collision operator, which acts only on the velocity dependence of $f$ quadratically:
\begin{equation}\label{collision-original}
\begin{split}
\Q(F,F)=&\int_{\mathbb{R}^3\times\mathbb{S}^{2}}(F'F'_*-FF_*) b(v-v_*,\omega) \bd v_*\bd\omega\,,
\end{split}
\end{equation}
where $F'=F(v')$, $F'_*= F(v'_*)$, $F_*= F(v_*)$ ($t$ and $x$ are only parameters), and the formulae
\begin{align*}
\begin{cases}
v' = \frac{v + v_*}{2} + \frac{|v-v_*|}{2}\omega \\
v_*'= \frac{v + v_*}{2} - \frac{|v-v_*|}{2}\omega\,,
\end{cases}
\end{align*}
yields a parametrization of the set of solutions to the conservation laws of elastic collision
\begin{align*}
\begin{cases}
v+ v_* = v' + v'_* \\
|v|^2 + |v_*|^2 =    |v'|^2 + |v'_*|^2\,.
\end{cases}
\end{align*}
Here $v$ and $v_*$ denote the velocities of two particle before the elastic collision, and $v'$ and $v_*'$ denotes the post-collision velocities. $\omega $ is equal to $\frac{v' -  v_*'}{| v' - v_*'|}$, belonging to $\mathbb{S}^2$ (the unit sphere in $\mathbb{R}^3$).  The nonnegative and a.e. finite weight function $b(v-v_*, \omega)$, called {  cross-section}, is assumed to depend only on the relative velocity $|v-v_*|$ and cosine of the derivation angle  $ (\frac{v-v_*}{|v-v_*|}, \omega)$. For a given interaction model, the cross section can be computed in a semi-explicit way by solving a classical scattering problem, see for instance,   \cite{cercignani1988book}. A typical example is that in dimension 3, for the inverse $s\mbox{-}$power repulsive forces (where $s > 1$ is the exponent of the potential), if denoted by $\kappa = \frac{v -  v_*}{| v - v_*|}$ and $\omega = \frac{v' -  v_*'}{| v' - v_*'|}$,
\begin{equation}\label{s-power-1}
b(v-v_*, \omega) = |v-v_*|^\gamma b(\kappa \cdot \omega) = |v-v_*|^\gamma b(\cos \theta)\,,\quad \gamma = \tfrac{s-5}{s-1}\,,
\end{equation}
and
\begin{equation}\label{s-power-2}
\sin\theta b(\cos\theta) \approx K \theta^{-1-s'}\quad\text{as}\quad\!\theta \to 0\,,\quad\text{where}\quad\! s'=\tfrac{2}{s-1}\quad\!\text{and}\quad\! K >0\,.
\end{equation}
Notice that, in this particular situation, $b(z,\omega)$ is not locally integrable, which is not due to the specific form of inverse power potential. In fact, one can show (see   \cite{villani2003book} ) that a non-integrable singularity arises if and only forces of infinite range are present in the gas. Thus, some assumptions must be made on the cross section to make the mathematical treatment of the Boltzmann equation convenient.

We first prove the existence of renormalized solutions.   For this purpose, the assumption on the cross section is the same as DiPerna and Lions in \cite{diperna-lions1989cauchy}, i.e. Grad's angular cutoff, namely, that {the cross section be integrable}, locally in all variables. More precisely, they assumed
\begin{equation}\label{Grad-cutoff-1}
A(z)= \int_{\mathbb{S}^{2}}b(z,\omega)\,\bd \omega \in L^1_{loc}(\mathbb{R}^3)\,,
\end{equation}
together with a condition of mild growth of $A$:
\begin{equation}\label{Grad-cutoff-2}
(1+|v|^2)^{-1}\int_{|z-v|\leq R}A(z)\,\bd z\to 0\quad\text{as}\quad\! |v|\to \infty\,,\quad\text{for all}\quad\! R<\infty\,.
\end{equation}

In the second part of the paper, we will study the incompressible NSF limit from the Boltzmann equation, for which in addition to the Grad's cutoff and the assumption \eqref{Grad-cutoff-2}, we need more restrictions on the cross-section. Since we justify the limit in the framework of \cite{lm2010soft} for the interior part, we make the same assumptions on the kernel as in \cite{lm2010soft}. For the convenience of the readers, we list as follows: the cross section $b$ satisfies
\begin{itemize}
	\item Assume that $\hat{b}$ has even symmetry in $\omega$,
	\begin{align}
	\label{lm-assump-1}
	\begin{cases}
	0 \le  b(z,\omega) = |v-v_*|^\gamma \hat{b}(\kappa \cdot \omega),~a.e. ~~\text{for} ~~-3 < \gamma \le 1,~\\
	\int_{\mathbb{S}^2} \hat{b}(\kappa \cdot \omega) \dd \omega < +\infty.
	\end{cases}
	\end{align}
	\item Attenuation assumption. Let $a(v)= \int_{\mathbb{R}^3} \bar{b}(v - v_* ) \mm(v_*)\dd v_*$. There exists some constant $C_a >0$ such that
	\begin{align}
	C_a ( 1 + |v|)^\alpha \le a(v), ~~\alpha \in \mathbb{R}.
	\end{align}
	\item Loss operator assumption.  Assume that there exists $s \in (1, +\infty]$ such that
	\begin{align}
	\label{lm-assump-4}
	\sup_{v \in \mathbb{R}^3}\bigg( \int_{\mathbb{R}^3} \vert \frac{\bar{b}({v - v_*})}{a(v) a(v_*)}\vert^s a(v_*) \mm(v_*) \dd v_*\bigg)^{\frac{1}{s}} < +\infty.
	\end{align}

	\item Gain operator assumption. The gain operator is given by
	\begin{align}
	\label{lm-assump-5}
	\mathcal{K}^+(g) = \frac{1}{2a(v)}\int_{\mathbb{S}^2\times\mathbb{R}^3}(g' + g_*)b(v-v_*,\omega)\dd \omega\mm(v_*)\dd v_*.
	\end{align}
	The gain operator assumption requires that $\mathcal{K}^+$ is a compact operator from $L^2(a\mm \dd v)$ to $L^2(a\mm \dd v)$.
	
\end{itemize}

\subsection{Incoming boundary condition}

Let $\Omega$ be an open and bounded subset of $\mathbb{R}^3$ and set $\mathcal{O}= \Omega \times \mathbb{R}^3$ and $\mathcal{O}_T=(0,T)\times\mathcal{O}$. We assume that the boundary $\partial\Omega$ is sufficiently smooth.
The regularity that we need is that there exists a vector field $\mathrm{n}\in W^{2,\infty}(\Omega\,;\mathbb{R}^3)$ such that $\mathrm{n}(x)$ coincides with the outward unit normal vector at $x\in\partial\Omega$. We define the outgoing and incoming sets $\Sigma_+$ and $\Sigma_-$ at the boundary $\partial\Omega$ as $\Sigma_{\pm} = \{(x,v)|x\in \partial\Omega\,, v\in \mathbb{R}^3\,, \pm v\cdot\mathrm{n}(x) > 0\}$, and $\Sigma = \Sigma_\pm \cup \Sigma_0 $.
We also denote by  $\mathrm{d}\sigma_x $ the Lebesgue measure on  ${\partial\Omega}$ and $\dd \mu = |\mathrm{n}(x) \cdot v| \dd v$.

The boundary condition considered in this paper is that the number density on the incoming to the domain is prescribed. More precisely, denoted by
$\gamma F$ be the trace of the number density (provided the trace can be defined), and let $\gamma_{\pm}F = \mathbf{1}_{\Sigma_{\pm}}\gamma F$. The so-called incoming boundary condition is that
\begin{align}\label{boundary-conditions-incoming}
\gamma_-F = Z\,,
\end{align}
where $Z \geq 0$ is a non-negative measurable function and satisfies
\begin{equation}\label{g-bound}
\int^T_0\int_{\Sigma_-} Z(1+ |v|^2 + |\log Z|)\,\bd v\bd \sigma_x\bd t < \infty\quad \text{for any}\quad\! T>0\,.
\end{equation}
The first question of this paper is to prove the existence of global-in-time renormalized solution of the Boltzmann equation with Grad's angular cutoff assumption on the cross-section and \eqref{Grad-cutoff-2} with the incoming data \eqref{boundary-conditions-incoming}-\eqref{g-bound}, and the initial data \eqref{IC-1}-\eqref{IC-bounds}.

\subsection{Incompressible Navier-Stokes limits}

The second part of this paper will focus on the incompressible Navier-Stokes limit. We start from the following scaled Boltzmann equation:
\begin{equation}\label{be-f}
\begin{cases}
\epsilon\partial_t F_\epsilon + \vdot F_\epsilon = \frac{1}{\epsilon}\Q(F_\epsilon,F_\epsilon)\,,\\
\gamma_- F_\epsilon = Z_\eps,\\
F_\epsilon(0,x,v) = F_\epsilon^0(x,v)\geq 0\,,
\end{cases}
\end{equation}
where $\eps>0$ is the Knudsen number. We take the Navier-Stokes scaling, i.e. the Boltzmann equation is scaled in the form of \eqref{be-f} and the solution is written as
\begin{align}
\label{ns-scaling}
F_\epsilon = \mm(1 + \epsilon g_\epsilon)\,.
\end{align}
Under this scaling, it can formally derive the incompressible NSF system with the corresponding boundary conditions, see \cite{bgl1991formal} and \cite{sone2007} for the derivations of the equations and boundary conditions respectively. More specifically, it can be shown formally that $g_\eps$ converges to $g$ as the Knudsen number $\eps\rightarrow 0$, and $g$ must have the form of the infinitesimal Maxwellian:
\begin{equation}
  g(t,x,v)= v\cdot \uu(t,x) + \theta(t,x)(\tfrac{|v|^2}{2}- \tfrac{5}{2})\,,
\end{equation}
where $(\uu,\theta)$ obeys the incompressible Navier-Stokes-Fourier (NSF) system
	\begin{align}
	\begin{cases}
	\partial_t \uu + \uu\cdot\nabla \uu + \nabla p - \nu \Delta \uu =0,\\
	\mathrm{div} \uu =0,\\
	\partial_t \theta + \uu \cdot  \nabla \theta -  k \Delta \theta = 0\,.
	\end{cases}
	\end{align}
Furthermore, to derive the boundary conditions for $(\uu,\theta)$, the incoming boundary data must have the form: $Z_\eps = \mm (1+\eps z_\eps)$, where ${z}_\eps$ satisfies formally:
\begin{equation}\label{BC-formal-con}
  z_\eps \rightarrow \rho^w + v\cdot \uu^w + \theta^w(\tfrac{|v|^2}{2}-\tfrac{3}{2})\,,
\end{equation}
then $(u,\theta)$ satisfies the boundary conditions: there exists two positive constants $\alpha, \beta$ with $\alpha + \beta >0$,
\begin{equation}\label{NSF-BC-1}
  \begin{cases}
    \alpha (\uu\!\cdot\! \mathrm{n}) - \theta  = \rho^w\,,\\
    \beta (\uu\!\cdot\! \mathrm{n}) + \theta  = \theta^w\,,\\
    \uu^{\mathrm{tan}}= (\uu^w)^{\mathrm{tan}}\,.
  \end{cases}
\end{equation}
The boundary conditions \eqref{NSF-BC-1} are Dirichlet conditions. In \eqref{NSF-BC-1}, the constants $\alpha\,,\beta$ are determined by the solvability of the linear kinetic boundary layer equation investigated by Bardos-Caflisch-Nicolaenko \cite{bcn1986}. In particular, if $(\rho^w\,, \uu^w\,, \theta^w)= (0,0,0)$, then $(\uu,\theta)$ satisfies the homogeneous Dirichlet boundary conditions, i.e. $\uu=0$ and $\theta=0$ on $\partial\Omega\,.$

To make the above formal analysis rigorous is the main concern of the second part of the current work. The justification of the NSF is the same as \cite{lm2010soft}, so we will omit the details in this paper and focus on the justification of the boundary conditions \eqref{NSF-BC-1}. To justify the non-homogeneous Dirichlet conditions \eqref{NSF-BC-1}, i.e. $(\rho^w\,, \uu^w\,, \theta^w)\neq (0,0,0)$ is a challenging problem. In this paper, we treat the homogeneous case. From the point view of formal convergence \eqref{BC-formal-con}, this needs the fluctuations of the incoming data $\hat{z}_\eps$ is very ``small" so that their limit is zero. We will characterize this smallness of the incoming data by the so-called incoming boundary relative entropy. We call the boundary data is ``well-prepared" if the incoming boundary relative entropy is of order $O(\eps^3)$, see \eqref{closetom}. The second main result of this paper is the justifications of NSF with homogeneous Dirichlet boundary for well-prepared incoming boundary data.

\section{Preliminaries and Main results}
\label{sec-notation}
In this section, we will introduce the definition of traces to solutions of transport equation, then some useful lemmas from \cite{mischler2000vlasovtrace,mischler2010asens} on weak-weak convergence.   As mentioned before, the solution $F$ only belongs to $L^1$ space. So its trace can not be defined in the usual way. In fact, the trace of $F$ can be defined in the weak sense or by employing characteristic line, for instance, see \cite{cercignani1992initial,hamdache1992initial,mischler2000vlasovtrace,mischler2010asens} and references therein.   If the solutions to transport equation are smooth, then their traces in weak sense are the same to these in usual sense. In this work, we mainly adopt Mischler's idea to define traces.

First, we introduce some notations.     $\mathcal{D}((0,T)\times\mathcal{O})$  is made up of   smooth function $\phi$ with compact support satisfying
\[ \phi(0, x, v) = \phi(T, x, v) = 0, \qquad \text{for all}~~ (x,v)\in \mathcal{O},  \]
\[ \phi(t, x, v)|_{\partial\Omega} = 0, \qquad \text{for all}~~ (t, v)\in (0,T)\times \mathbb{R}^3. \]
$\mathcal{D}([0,T]\times\bar{\Omega}\times\mathbb{R}^3)$ contains smooth functions which have compact support and may not vanish on the boundary. In the similar way, we can define $\mathcal{D}((0,T)\times\bar\Omega)$.

 Let $\mathbf{P}$  be   Leray   projection in $L^2(\Omega)$ onto its subspace  with divergence-free vector. Denote  by $\mathcal{L}$ the linear Boltzmann collision operator(linearized around $\mm$)      given by
\begin{align*}
\mathcal{L}g:= \int_{\mathbb{R}^3\times\mathbb{S}^{2}}(g - g' + g_* - g_*') b(v-v_*,\omega) \mm(v_*)\bd v_*\bd\omega.
\end{align*}
It is a Fredholm operator.  Its  kernel  space is spanned by linear independent vectors $1$, $v$ and $\frac{|v|^2 - 3}{2}$. Specially,
the kinetic momentum flux and heat flux \begin{align*}
\mathrm{A}(v) = v \otimes v -\tfrac{|v|^2}{3},~~\mathrm{B}(v) = v (\tfrac{|v|^2}{2} - \tfrac{5}{2})
\end{align*}
lies in the $\mathrm{Ker}^\perp(\mathcal{L})$. Thus there exists $\hat{\mathrm{A}}$　 and $\hat{\mathrm{B}}$ such that
\begin{align*}
\mathcal{L}\hat{\mathrm{A}} = {\mathrm{A}}, ~ \mathcal{L}\hat{\mathrm{B}} = {\mathrm{B}}.
\end{align*}
We write  $L^1(\mathcal{O};\dd v \dd x)$ as $L^1(\mathcal{O})$,  $L^1((0,T)\times\Sigma_\pm;\xmus)$ as $L^1((0,T)\times\Sigma_\pm)$ for short.

\begin{lemma}[Green Formula {\cite{mischler2010asens}}]
	\label{green-formula-Lp}
	Let $p\in[1, +\infty)$,  $ g \in L^\infty((0,T), L^p_{loc}({\mathcal{O}})$ and $ h \in L^1 ((0,T), L^p_{loc}({\mathcal{O}})$. Assume that $g$ and $h$ satisfy equation
	\begin{align*}
	\partial_t g + v \cdot \nabla_x g = h,
	\end{align*}
	in distributional sense. Then there exists  $\gamma g$ well defined on $(0,T) \times \Sigma$  which satisfies
	\[  \gamma  g \in L^1_{loc} \big([0,T]\times\Sigma, (\mathrm{n}(x)\cdot v)^2 \bd v \bd\sigma_x \dd t\big), \]
	and the following Green Formula
	\begin{align*}
	&\int_{0}^{T} \int_{\mathcal{O}} \big( \beta(g)(\partial_t \phi  + v \cdot \nabla_x \phi) + h \beta'(g)\phi\big) \mathrm{d}v \mathrm{d}x \mathrm{d}t \\
	& = \big[ \int_{\mathcal{O}} \phi(t,\cdot)\beta(g)(\tau, \cdot) \mathrm{d}v \mathrm{d}x \big]\vert_{0}^{T} + \int_{0}^{T} \int\int_{\Sigma} \phi \beta(\gamma g) \bd \mu \bd \sigma_x  \bd t,
	\end{align*}
	for    $\beta(\cdot) \in W^{1,\infty}_{loc}(\mathbb{R}_+)$ with $\sup_{x\ge 0}|\beta'(x)| < \infty$,   and  $\phi \in \mathcal{D}({[0,T]\times\bar{\Omega}\times\mathbb{R}^3})$.
\end{lemma}
\begin{remark}
	\label{remark-renormlized}
	According to the definition, we can find
	\[ \gamma \beta(F) = \beta (\gamma F).  \]
\end{remark}
While the solution belongs to $L^\infty$ space, its trace also belongs to the corresponding $L^\infty$ space. Besides, the trace also enjoys monotone properties.
\begin{lemma}[Green Formula of $L^\infty$]
	\label{green-formula-L-infty-defect}~\\
	(I), Let  $ g \in L^\infty((0,T)\times{\mathcal{O}})$ and  $ h \in L^1((0,T)\times\mathcal{O}) $. Assume that $g$ and $h$ satisfies equation
	\begin{align}
	\label{equation-defect}
	\partial_t g + v \cdot \nabla_x g = h ,
	\end{align}
	in distributional sense. Then there exists  $\gamma g \in L^1_{loc} \big([0,T]\times\Sigma, (\mathrm{n}(x)\cdot v)^2 \dd v \dd \sigma_x \dd t\big)$   which satisfies  the following Green Formula
	\begin{align*}
	&\int_{0}^{T} \int_{\mathcal{O}}   \beta(g)(\partial_t \phi  + v \cdot \nabla_x \phi)  + h \beta'(g) \phi \mathrm{d}v \mathrm{d}x \mathrm{d}t  \\
	& = \big[ \int_{\mathcal{O}} \beta(g)(\tau, \cdot) \phi \mathrm{d}v \mathrm{d}x \big]\vert_{0}^{T} + \int_{0}^{T} \int_{\Sigma} \beta(\gamma g) \mathrm{d}\mu \mathrm{d}\sigma_x  \mathrm{d}t,
	\end{align*}
	for all $\beta(\cdot ) \in W^{1,\infty}_{loc}(\mathbb{R}_+)$ with $\sup_{x\ge 0}|\beta'(x)| < \infty$ and   $\phi \in \mathcal{D}( {[0,T]\times\bar{\Omega}\times\mathbb{R}^3})$.
	\par
	(II), Moreover,  assume that there are $g_1$ and $g_2$, $h_1$ and $h_2$ such that
	\[ \partial_t g_1 + v \cdot \nabla_x g_1 = h_1, \]
	and
	\[  \partial_t g_2 + v \cdot \nabla_x g_2 = h_2, \]
	hold in distributional sense. If $  C_2 \ge g_1 \ge g_2 \ge 0$, then
	\[ C_2 \ge \gamma_\pm g_1 \ge \gamma_\pm g_2 \ge 0.\]
\end{lemma}	
\begin{proof}This lemma can be proved by the argument in \cite{mischler2000vlasovtrace,mischler2010asens}. Specially, $(I)$ is the same as the one in \cite{mischler2000vlasovtrace}.

Choosing the mollifer $\rho_\epsilon$ such that	
\[ \rho_n(x, v) = \tfrac{1}{n^3}\rho(\tfrac{v}{ n}) \tfrac{1}{n^3}\rho(\tfrac{x}{ n} - \mathrm{n}(x) \cdot \tfrac{2}{n}),\]
with $\mathrm{Supp}\rho \subset \{ x\in \mathbb{R}^3 \vert~~ |x| <1 \}$ and $\int_{\mathbb{R}^3} \rho(x) \bd x = 1$.
	
	 Indeed, applying    $\rho_{n}( x, v)$  to \eqref{equation-defect}, then
	\[  \partial_t \rho_{{n}}* g + v \cdot \nabla \rho_{{n}}*g=  \rho_{{n}}*h + r_{{n}}(g),   \]
	where $r_{{n}}(g) \to 0$ in $L^1((0,T)\times\Omega)$ as $n$ goes to infinity.
	
	Now for very $n \in \mathbb{N}^+$, $\rho_{n}*g$ is smooth. The trace of $ \rho_n * g$ on $\Sigma_\pm$ have a clear meaning in the usually sense.
	Indeed, for all $(x, v) \in \Sigma_\pm$,
	\[ \gamma( \rho_{n}*g)(t, x, v) = \lim\limits_{y \to x\atop \xi \to v} (\rho_{n}*g)(t, x, v),~~ (x, \xi ) \in \mathcal{O}.  \]
		So if $ C_2 \ge g_1 \ge g_2 \ge 0$, then
	\begin{align}
	\label{est-lemma-trace-l-infty-mono-trace}
	C_2\ge \gamma ( \rho_{n}*g_1 )(t, x, v)  \ge (\gamma  \rho_{n}*g_2 )(t, x, v),~~(x, v) \in \Sigma_\pm.
	\end{align}
	
		 Setting $ \gamma g_{{i}} = \gamma (  \rho_{n}*g_i ), r_{n,i}=r(g_i)(i=1,2)  $,    by Stokes formula, for any $ 0 < t \le T$
	\begin{align*}
	&\int_{0}^{t} \int_{\mathcal{O}} \bigg( \beta( \rho_{n} * g_i)(\partial_t \phi  + v \cdot \nabla_x \phi) +  r_{{n,i}} \beta'( \rho_{n} * g_i) \phi\bigg) \mathrm{d}v \mathrm{d}x \mathrm{d}\tau \\
	& = \big[ \int_{\mathcal{O}} \beta( \rho_{n} * g_i)(\tau, \cdot) \phi \mathrm{d}v \mathrm{d}x \big]\vert_{0}^{t_1} + \int_{0}^{T}\int_{\Sigma} \beta\big(\gamma (  \rho_{n}*g_i ) \big) \phi \dd \mu \dd \sigma_x  \dd \tau.
	\end{align*}

By the same argument in \cite{mischler2000vlasovtrace}, there exists $\gamma g \in  L^1_{loc} \big([0,T]\times\Sigma_\pm, (\mathrm{n}(x)\cdot v)^2 \bd v \bd \sigma_x \bd t$ such that
\[ \gamma  \rho_{n}* g \to \gamma   g,~ \text{in}~ L^1_{loc} \big([0,T]\times\Sigma, \dd \mu \bd \sigma_x \bd t. \]
This hints that
\[ \gamma  \rho_{n}* g \to \gamma   g, ~a.e. ~\text{on}~~(0,T)\times\Sigma. \]
By \eqref{est-lemma-trace-l-infty-mono-trace} and Lebesgue dominated convergence theorem, We complete the proof.
\end{proof}

Recalling that $\Q(F,F)$ does not belong to $L^1$ space,  $\beta'(F)\Q(F,F)$($\beta(\cdot ) \in C^{1}(\mathbb{R}_+)$ with $\sup_{x\ge 0}(1 + x)|\beta'(x)| < \infty$)   belongs to $L^1$ space. So we can only apply Lemma \ref{green-formula-Lp} to renormalized version of Boltzmann equation, namely the resulting equation obtained  by multiplying \eqref{BE} by $\beta'(F)$. So the trace is defined just for $\beta(F)$ other than $f$ itself. The following {\it r-convergence} is very useful while  recovering  $\gamma f$ from  $\gamma \beta(F)$.

Let $X$ and $Y$ be a separable and $\sigma-$compact topological space. We denote by $g^n \uparrow g$ that $g^n$ converges increasingly to $g$ in some function space to be clear nearby.
\begin{definition}[\cite{mischler2010asens}]
	\label{def-r-convergence}
	We say that  $\alpha$ is a renormalizing function if $\alpha \in C(\mathbb{R})$ is increasing and $ 0 \le \alpha(s) \le s$ for any $s \ge 0$. We say that $\{\alpha_m\}$ is renormalizing sequence if for any $m\in \mathbb{N}^+$,  $\alpha_m$ is a renormalizeing function, $\alpha_m(s) \le m$  and $\alpha_m \uparrow s$ for all $s \ge 0$ when $m \uparrow \infty$. Given any renormalizing sequence $(\alpha_m)_{m\in \mathbb{N}^+}$, we say that $Z_n$  r-converges(r-convergence) to $Z$ if there exists a sequence $\{\bar{\alpha}_m\}$ in $L^\infty(Y)$ such that
	\[ \alpha_m(Z_n) \rightharpoonup \bar{\alpha}_m,~~\text{in} ~L^\infty(Y);~~\text{and}~~ \bar{\alpha}_m \uparrow Z,~~a.e.   \text{in} ~~Y.\]
\end{definition}
We denote r-convergence by $Z_n\rcong Z$.
\begin{proposition}~
	\label{proposition-r-convergence}
	\begin{itemize}
		\item If $Z_n \rcong Z$, then $Z$ is independent of the renormalizing sequence $(\alpha_m)_{m\in\mathbb{N}^+}$. This hints that one can choose any renormalizing sequence as you like as long as it is a renormalizing sequence.
		\item If $ h^n \rightharpoonup h$ in $L^1(X)$, then $h^n \rcong h$ in $ X$. \cite[Lemma 2.7]{mischler2010asens}
	\end{itemize}
\end{proposition}

Now, we introduce  the definition of renormalized solutions to Boltzmann equation with incoming boundary condition.
\begin{definition}
	\label{def-renormalized-cutoff}
	Let  $\beta(x)\in C^1(\mathbb{R}_+)$ satisfy   $ \beta'(x) \ge 0, x \ge 0$ and $  \sup_{x\ge 0} ( 1 +x)\beta'(x)  < \infty $.  Assume that the cross section   satisfies \eqref{Grad-cutoff-1} and \eqref{Grad-cutoff-2}.  A nonnegative function
	\[ F \in C\big(\mathbb{R}_+, \mathcal{D}'(\mathcal{O})\big) \cap L^\infty(\mathbb{R}_+; L^1((0,T)\times\mathcal{O}; \mathrm{d}v \mathrm{d}x )\big)\]
	is  a {\it renormalized solution} to initial boundary problem  of Boltzmann equation \eqref{BE} with initial datum \eqref{IC-bounds} and boundary condition \eqref{g-bound} in the distributional sense: if
	\[ \beta'(F)\Q(F,F) \in L^1((0,T)\times\mathcal{O}; \mathrm{d}v \mathrm{d}x \mathrm{d}s)  \]	
	and  there exist traces $\gamma_+ F   \in L^1((0,T)\times\Sigma; \mathrm{d}\mu \mathrm{d}x \mathrm{d} s)$  such that
	
	\begin{align}
	\label{est-definition-trace-of-renormalized-cutoff}
	\begin{split}
	&\int_{0}^{T} \int_{\mathcal{O}} \big( \beta(F)(\partial_t \psi  + v \cdot \nabla_x \psi) + \Q(F, F)\beta'(F) \psi\big) \mathrm{d}v \mathrm{d}x \mathrm{d}t \\
	& =  \int_{\mathcal{O}} \beta(F)(T, \cdot) \psi(T)  \mathrm{d}v \mathrm{d}x  - \int_{\mathcal{O}} \beta(F_0)\psi(0)  \mathrm{d}v \mathrm{d}x \\
	& + \int_{0}^{T}  \int_{\Sigma_+} \beta(\gamma_+ F) \psi \mathrm{d}\mu \mathrm{d}\sigma_x \mathrm{d}t - \int_{0}^{T}  \int_{\Sigma_-} \beta(Z) \psi \mathrm{d}\mu \mathrm{d}\sigma_x \mathrm{d} t,
	\end{split}
	\end{align}
	for  any test function $\psi \in \mathcal{D}([0,T]\times \bar{\Omega}\times\mathbb{R}^3)$.
\end{definition}

\subsection{Existence of renormalized solutions}

The relative entropy $ H(F | \mathrm{M} )$ is defined as
\[ H(F | \mathrm{M} ) = \int_\mathcal{O} h(F/\mathrm{M}) \mm  \bd v \bd x,   \]
where
$$ h(x) = x\log x - x + 1\,,\quad\mbox{for}\quad\! x \ge 0\,.$$
We also denote by $\mathcal{D}(F)$ the {\it H-dissipation}
\[\small 4\mathcal{D}(F) = \int_\Omega \int_{\mathbb{R}^3 \times\mathbb{R}^3}\int_{\mathbb{S}^2}  B(v-v_*, \omega)(F'F_*' - FF_*)\log \tfrac{F'F_*'}{FF_*} \bd \omega \bd v \bd v_* \bd x .\] Our main result is as follows.
\begin{theorem}
	\label{main-result-cut}
	Under the assumption \eqref{Grad-cutoff-1} and \eqref{Grad-cutoff-2} on cross section, if the initial datum $F_0$ satisfies \eqref{IC-1} and the boundary condition $z$ satisfies \eqref{g-bound},
	then  the initial-boundary problem to Boltzmann equation \eqref{BE}  admits a global renormalized solution $F \in L^1 \cap L \log L$.

	Furthermore, $F$ has the following properties:
	\begin{itemize}
		
			\item Local   conservation law of mass:
		\begin{align}
		\partial_t \int_{\mathbb{R}^3} F(t)\mathrm{d} v + \nabla \cdot \int_{\mathbb{R}^3} F(t) v \mathrm{d} v = 0,~ \text{in}~~ \mathcal{D}'((0,T)\times\Omega).
		\end{align}
		
		\item Estimate of traces
		\begin{align}
		\label{est-theorm-exi-trace-1}
		\gamma_\pm F \in L^1\big((0,T)\times\Sigma_\pm; (1+ |v|^2)\dd \mu \mathrm{d}\sigma_x \bd s\big), ~ \text{for ~ all~} T>0,
		\end{align}
		and
		\begin{align}
		\label{est-theorem-exi-trace-2}
		\int_{0}^{T} \int_{\Sigma_\pm}\gamma_\pm F |\log \gamma_\pm F| \mathrm{d} \mu \mathrm{d}\sigma_x \mathrm{d}s < +\infty,~\text{for ~ all~} T>0.
		\end{align}
		\item Commutative properties:
		As for the local  conservation law of mass, similar to Lemma \ref{green-formula-Lp},  we can use the Green formula to define the trace of $\int_{\mathbb{R}^3} v F\mathrm{d} v$ on $\partial\Omega$,   Denoting it by $\gamma_x(\int_{\mathbb{R}^3} v F\mathrm{d} v)$.  Moreover
		\begin{align*}
		\mathrm{n}(x) \cdot  \gamma_x(\int_{\mathbb{R}^3} v F\mathrm{d} v)  = \mathrm{n}(x) \cdot  \int_{\mathbb{R}^3} v \gamma F\mathrm{d} v.
		\end{align*}
		This means that the trace operator $\gamma$ is commutative with integral operator $\int$. This is because that the trace of solutions enjoys the full estimate.
		
		\item Local  conservation law of momentum: There is a distribution-value matrix $\mathbf{W}$ belonging to $\mathcal{D}'((0,T)\times\Omega)$   such that
		\begin{align*}
		\partial_t \int_{\mathbb{R}^3} v F(t)\mathrm{d}v + \nabla \cdot   \int_{\mathbb{R}^3} v\otimes v F(t)  \mathrm{d}v  + \nabla \cdot \mathbf{W}  =0 ,~~\text{in}~~~ \mathcal{D}'((0,T)\times\Omega).
		\end{align*}
		\item Global  conservation law of momentum:
		\begin{align}
		\begin{split}
		&\int_{\mathcal{O}} F(t) v \mathrm{d}v \mathrm{d}x + \int_{0}^{t} \int_{\Sigma_+} v \gamma_+ F(s) \dd \mu \mathrm{d}\sigma_x \mathrm{d}s  \\
		&  = \int_{\mathcal{O}} F_0  v \mathrm{d}v \mathrm{d}x + \int_{0}^{t} \int_{\Sigma_-} v  Z(s) \dd \mu \mathrm{d}\sigma_x \mathrm{d}s, \qquad t \le T.
		\end{split}
		\end{align}
		
		\item Global energy  inequality:
		\begin{align}
		\begin{split}
		&\int_{\mathcal{O}} F(t) |v|^2  \mathrm{d}v \mathrm{d}x + \int_{0}^{t} \int_{\Sigma_+} |v|^2 \gamma_+ F(s) \dd \mu \mathrm{d}\sigma_x \mathrm{d}s  \\
		&  \le \int_{\mathcal{O}} F_0 |v|^2 \mathrm{d}v \mathrm{d}x + \int_{0}^{t} \int_{\Sigma_-}  Z(s) |v|^2 \dd \mu \mathrm{d}\sigma_x \mathrm{d}s, \qquad t \le T.
		\end{split}
		\end{align}	
		
		\item Global entropy inequality:
		\begin{align}
		\label{est-theorem-main-entropy}
		\begin{split}
		& H(F| \mathrm{M} )(t) + \int_{0}^{t}\int_{\Sigma_+}h(\gamma_+ F/\mathrm{M} )(s) \dd \mu \mathrm{d}\sigma_x\mathrm{d}s +  \int_{0}^t \mathcal{D}(F)(s) \mathrm{d}s  \\
		&\le  H(F_0| \mathrm{M} )  + \int_{0}^{t}\int_{\Sigma_-}h(Z| \mathrm{M} )(s) \dd \mu \mathrm{d}\sigma_x\mathrm{d}s,\qquad t \le T.
		\end{split}
		\end{align}
	\end{itemize}
\end{theorem}

\begin{remark}
	\label{remark}
	First, compared to Maxwell reflection boundary condition,  the estimates \eqref{est-theorm-exi-trace-1} and \eqref{est-theorem-exi-trace-2} contain full informations of incoming set and out going set other than partial estimates.  Furthermore, as a consequence, the trace operator $\gamma$ is commutative with the integral operator.

\end{remark}

\begin{remark}
	This result also works for   unbounded domain case. While   on the unbounded domain, the weight $|x|^2$ are necessary.  Besides, all these result are still correct in $\mathbb{R}^n$, $n\ge 2$.
\end{remark}
Now we state the difficulty and strategy. As mentioned before, the solution obtained in Theorem \ref{main-result-cut} only makes sense in the renormalized sense, namely $f$ is a solution to  the renormalized version of \eqref{BE}, namely
\[ \partial_t \beta(F) + v \cdot \nabla_x \beta(F) = \beta'(F)\Q(F,F).  \]
Of course, we can use Lemma \ref{green-formula-Lp} to get the existence of traces. Noticing that $\Q(F,F)$ does not belong to $L^1$ space, while we only have $\beta'(F)\Q(F,F)$ in $L^1$ space. So we can only use Lemma \ref{green-formula-Lp} to the renormalized version of \eqref{BE}, namely we get the existence of $\gamma_\pm \beta(F)$. It is not  easy  to recover the trace of $F$ satisfying Theorem \ref{main-result-cut} from the trace of $\beta(F)$ while $F$ only belongs $L^1(\mathcal{O})$. Furthermore, since the trace of $F$ on $\Sigma_-$ is fixed, the main goal is to find some $\gamma_+ F$ satisfying  \eqref{est-definition-trace-of-renormalized-cutoff}. Motivated by \cite{diogo2012,mischler2010asens}, we  choose a sequence of $\{\beta_j\}$ with $ \beta_j(x) = \frac{jx}{ j + x} $ to fulfill this. By Lemma \ref{green-formula-Lp}, we can get $\gamma_\pm \beta_j(F)$, the trace of $\beta_j(F)$. For any $\beta$ satisfying the assumption in Def. \ref{def-renormalized-cutoff},
\[ \beta(\beta_j(f)) \to \beta(F),~~j \to \infty, \text{in}~ L^1(\mathcal{O}),  \]
we will show that $(\gamma_+ \beta_j(F))$ is a monotone sequence and thus there exists some $\gamma_+ F$ such that
\[  \gamma_+ \beta_j(F) \uparrow \gamma_+ F,~~\text{on}~~\mathcal{O}, \]
and $\gamma_+ F$  satisfies \eqref{est-definition-trace-of-renormalized-cutoff}.
  Moreover we still find that the trace operator are commutative with integration operator, namely Remark \ref{remark}. It  is unknown whether Remark \ref{remark} is    true for the Maxwell reflection boundary conditions, \cite{mischler2010asens}. This is a new discovery.

\subsection{  Navier-Stokes limit}
We define
\[ G_\epsilon^0 = F_\epsilon^0/\mm,~~G_\epsilon  = F_\epsilon/\mm . \]
Noticing that if the kernel satisfy assumptions \eqref{lm-assump-1} to \eqref{lm-assump-5}, then it   satisfies \eqref{Grad-cutoff-1} and \eqref{Grad-cutoff-2}. The existence of renormalized solutions (including their traces) to scaled Boltzmann equation \eqref{be-f} is ensured by Theorem \ref{main-result-cut}. Indeed, there exists a family of renormalized solutions $F_\eps \in L \log L$ satisfying
		\begin{align}
\label{est-f-eps-entropy}
\begin{split}
& \eps H(G_\eps )(t) + \int_{0}^{t}\int_{\Sigma_+}h(\gamma_+ G_\eps ) \dd \mu \mathrm{d}\sigma_x\mathrm{d}s + \tfrac{1}{\eps } \int_{0}^t \mathcal{D}(F_\eps) \mathrm{d}s \\
& \le  \eps  H(G_\eps^0 ) + \int_{0}^{t}\int_{\Sigma_-}h(z_\eps ) \dd \mu \mathrm{d}\sigma_x\mathrm{d}s.
\end{split}
\end{align}
\begin{theorem}
	\label{main-result}
	Under the assumptions from \eqref{lm-assump-1} to \eqref{lm-assump-5} on cross section, for every $\eps>0$, let $F_\epsilon$ be DiPerna-Lions renormalized solutions to Boltzmann equation \eqref{be-f}.  Under the Navier-Stokes scaling \eqref{ns-scaling}, let  \begin{align}
	\label{be-f-initial-eps}
	C_0 = \sup_{\epsilon >0} \frac{1}{\epsilon^2}\int_{\Omega}\int_{\mathbb{R}^3}h(G_\epsilon^0)\mm \mathrm{d} v \mathrm{d}x < +\infty,
	\end{align}
	and $z_\eps$ be close to $\mm$ in the relative entropy sense:
	\begin{align}
	\label{closetom}
	C_1 = \sup_{\epsilon >0}\big( \frac{1}{\eps^3}\int_{0}^{+\infty}\int_{\Sigma_-}h(z_\eps ) \dd \mu  \mathrm{d}\sigma_x\mathrm{d}s \big) < + \infty.
	\end{align}
Assume that for some $(\mathrm{u}_0, \theta_0) \in L^2(\Omega)$, $ \mathbf{P} \big( \tfrac{1}{\epsilon}\int v F_\epsilon^0 \dd v \big) $ and $\tfrac{1}{\epsilon} \int (\tfrac{|v|^2}{5} - 1) (f^0_\epsilon - \mm)  \dd v$ converges to $\mathrm{u}_0$ and $\theta_0$ respectively in the sense of distribution.
	Then, $g_\eps$  is weakly compact in $L_{loc}^1(\dd t\dd x; L^1(( 1 + |v|^2)\mm \dd v)$. Moreover, for every limiting point $g$, it is an infinitesimal Maxwellian with form
	\begin{align}
	\label{limit-g}
	g(t,x,v) = \mathrm{u}(t,x) \cdot v + \theta(t,x)\tfrac{|v|^2 - 5}{2}  ,
	\end{align}
	where  $(\uu,\theta) \in L^\infty((0,\infty); L^2(\Omega))$
	is a  {\it Leray }solution  to  NSF system with initial data
	{ \begin{align*}
		\uu(0,x) = u_0, ~~ \theta(0,x) = \theta_0(x),
		\end{align*}
		and  boundary conditions
		\begin{align*}
		\mathrm{u}|_{\partial\Omega} = 0,~~~~ \theta|_{\partial\Omega} = 0.
		\end{align*}
		Besides,  the viscosity coefficient $\nu$ and  thermal conductivity $k$ are defined as follows
		\begin{align*}
		\nu = \tfrac{1}{10} \int_{\mathbb{R}^3}   \mathrm{A}:\hat{\mathrm{A}}  \mm \dd v ,~~k = \tfrac{2}{15} \int_{\mathbb{R}^3}   \mathrm{B}:\hat{\mathrm{B}}  \mm \dd v .
		\end{align*} }
	and for any $t>0$, $u$ and $\theta$ satisfy the following energy estimates
	\begin{align}
	\label{est-u-theta-f}
	\int_\Omega \big(  \tfrac{1}{2}|u|^2 + \tfrac{5}{4} |\theta|^2\big)(t,x)\dd x + \int_0^t \int_\Omega\big( \nu |\nabla u|^2 + \tfrac{5k}{2} |\nabla \theta|^2 \big)(s,x) \dd x \dd s\le C_0 + C_1.~
	\end{align}	
Furthermore, for every subsequence $g_{\eps_k}$ of $g_\eps$ converging to $g$ in $L_{loc}^1(\dd t\dd x; L^1(( 1 + |v|^2)\mm \dd v)$ space as $\eps_k \to 0$, it also satisfies
\begin{align*}
\mathbf{P}\la v g_{\epsilon_k} \ra \to \uu,~~\la (\tfrac{|v|^2}{5} - 1) \tilde{g}_\eps   \ra \to \theta,~~ \text{in}  ~~ C([0,\infty)\,,  \mathcal{D}'(\mathbb{R}^3)).
\end{align*}		

\end{theorem}

\begin{remark}
	If $z_\eps = \mm$, the homogeneous incoming boundary condition, by simple calculation $C_1=0$.  From our result, if the given incoming boundary condition is close to $\mm$ in the relative entropy sense,  the homogeneous Dirichlet boundary conditions of Navier-Stokes equations can be derived.   Furthermore, according to  Sone's book \cite{sone2002book,sone2007}, the jump type boundary  conditions of NSF equations can be  derived   while $z_\eps$ is not close to $\mm$ in the relative entropy sense, i.e. $C_1 = + \infty$.   But this case  is more challenging. At least, there is no uniform relative entropy bound for fluctuations $G_\eps$ with respect to $\eps$. We leave it to the future work.
	
\end{remark}
For Navier-Stokes limit, the proof of the interior domain follows the proof of \cite{lm2010soft}. Here, we focus on the boundary part, that is to say, how to derive the boundary conditions of NSF   from homogeneous incoming boundary condition of Boltzmann equation.  The idea and strategy will be elaborated during its proof.

\section{Estimates of Approximate system}
\label{sec-approximate}
In this section, we will construct a sequence of approximate solution to Boltzmann equation with modified collision kernel $\Q^n$, namely
\begin{equation}
\label{collision-original-cut}
\begin{split}
\Q^n(F,F)=\big(\frac{1}{\footnotesize 1 + \frac{1}{n}\int F \bd v}\big)\int_{\mathbb{R}^3\times\mathbb{S}^2}B_n [F(v')F({v_*}')-F(v)F(v_*)]\bd v_*\bd \omega
\end{split}
\end{equation}
with \begin{align}
\label{cross-section}
B_n(v-v_*,\omega) = B(v-v_*, \omega) \cdot    \mathbf{1}_{ \frac{1}{n} \le  |v - v_*| \le n}.
\end{align}
For every $n \in \mathbb{N}^+$, the initial data approximate system  are chosen as these in \cite{diperna-lions1989cauchy}, namely \begin{align}
\label{app-f-0}
F_0^n = \tilde{F}_0^n + \tfrac{1}{n}\exp(-\tfrac{|x|^2}{2} - \tfrac{|v|^2}{2}),
\end{align}
where $\tilde{F}_0^n$ is obtained by truncating $f_0$ first and then smoothing it.  In details,   we will solve the following initial-boundary problem
\begin{align}
\label{boltzmann-approximate-lions}
\begin{cases}
\partial_t F^n + v \cdot \nabla_x F^n = \Q^n(F^n, F^n), \\
F^n(0,x, v) = F_0^n(x, v),\\
\gamma_-F^n =Z,~~\text{on}~~\Sigma_-
\end{cases}
\end{align}
where $Z$ satisfies  for all $t >0$
	\begin{align}
\label{app-g}
\int_{0}^{t} \int_{\Sigma_-} Z(1 + |v|^2 + |\log z|)  \mathrm{d}\mu \mathrm{d}\sigma_x \mathrm{d}s < C(t) < \infty,
\end{align}
and $F_0^n$ satisfies
	\begin{align}
	\label{app-f}
  \int_{\mathcal{O}} F_0^n(1 + |v|^2 + |\log F_0^n|)  \mathrm{d}\mu \mathrm{d}\sigma_x \mathrm{d}s < C_0 < \infty,~~\text{for all}~ n \in \mathbb{N}^+.
	\end{align}

For each fixed $n\in\mathbb{N}^+$,   System \eqref{boltzmann-approximate-lions} can be solved by the fixed point theorem to following iteration system
	\begin{align*}
	\begin{cases}
	\partial_t F^{n,k+1} + v \cdot \nabla_x F^{n,k+1} = \Q^n(F^{n,k}, F^{n,k}), \\
	F^{n,k}(0,x, v) = F_0^n(x, v),\\
	\gamma_-F^{n,k} = Z,~~\text{on}~~\Sigma_- .
	\end{cases}
	\end{align*}
	From \cite{diperna-lions1989cauchy},
	 \begin{align*}
	 \|Q^n(F, F)\|_{L^1({\mathcal{O}})} \le C_n \|F\|_{L^1({\mathcal{O}})},
	 \end{align*}
then for every $k$, the existence of the above equation is equal to that of the following equation:
\begin{align}
\label{boltzmann-approximate-linear}
\begin{cases}
\partial_t h + v \cdot \nabla_x h = \tilde{H},\\
h(0, x, v)=h_0(x, v),\\
\gamma_{-}h = z \quad \Sigma_{-}\, ,
\end{cases}
\end{align}
where for any $  T >0$,
\begin{align}
\label{linear-i-b}
 h_0 \in L^1(\mathcal{O}),~ \tilde{H} \in L^1((0,T) \times \mathcal{O}),~ z \in L^1((0,T)\times\Sigma_+).
 \end{align}

\par

\begin{theorem}[Existence of \eqref{boltzmann-approximate-linear}]
	\label{theorem-boltzmann-approximate-linear}
Under the assumption \eqref{linear-i-b},   there is a unique solution $h \in L^\infty([0,T], L^1(\mathcal{O}))$ such that
\[ \partial_t h + v \cdot \nabla_x h = \tilde{H}, \]
holds in the sense of distribution. Furthermore, there exists  a unique trace $\gamma_+ h \in L^1((0,T)\times\Sigma_+ ) $ to \eqref{boltzmann-approximate-linear}  such that
	\begin{align*}
	&\int_{0}^{T} \int_{\mathcal{O}} \big( \beta(h)(\partial_t \phi  + v \cdot \nabla_x \phi) + \tilde{H} \beta'(h)\phi\big) \mathrm{d}v \mathrm{d}x \mathrm{d}t \\
	& =   \int_{\mathcal{O}}  \phi(T) \beta(h)(T) \mathrm{d}v \mathrm{d}x   -    \int_{\mathcal{O}}  \phi(0) \beta(h_0)  \mathrm{d}v \mathrm{d}x   \\
	&  + \int_{0}^{T}  \int_{\Sigma_+} \phi \beta(\gamma_+ h) \bd \mu \bd \sigma_x  \bd t - \int_{0}^{T}  \int_{\Sigma_-} \beta( z) \phi  \bd \mu \bd \sigma_x  \bd t,
	\end{align*}
	for all $\beta(\cdot ) \in W^{1,\infty}_{loc}(\mathbb{R})$ with $\sup_{x \in \mathbb{R}}|\beta'(x)| < \infty$   and all the test function $\phi \in \mathcal{D}({[0,T]\times\bar{\Omega}\times\mathbb{R}^3})$.
Moreover, the solution and its trace satisfy the following estimate
\begin{align}
\label{estimate-of-linear}
&\int_{\mathcal{O}} |h|(t) \mathrm{d}v \mathrm{d}x + \int_{0}^{t}\int_{\Sigma_+} |\gamma_+ h| \mathrm{d}\mu \mathrm{d}\sigma_x\mathrm{d}s\qquad\qquad \nonumber\\
&\le C\bigg( \int_{{\mathcal{O}_T}} |\tilde{H}| \mathrm{d}v \mathrm{d}x \bd s + \int_{\mathcal{O}} |h_0| \mathrm{d}v \mathrm{d}x + \int_{0}^{t}\int_{\Sigma_-} | z| \mathrm{d}\mu \mathrm{d}\sigma_x\mathrm{d}s\bigg),\qquad t \le T.
\end{align}
	
\end{theorem}

\begin{proof}
	
	This equation can be solved by characteristic line method.  Indeed, for any   $s \ge 0$ and any molecule which   locates at $ x \in \Omega$ with velocity $v$, we can define its trajectory as
	\[ S_t(s, x, v) = (s + t, x + vt, v),~~ S_t \in \mathcal{O},~ -t_1^- < t < t_1, x_t = x + vt, \]
	where $t_1$ denotes the maximal forward time of staying in $\mathcal{O}$(with velocity $v$) and $t_1^-$ denotes the maximal backward time of staying in $\mathcal{O}$(with velocity $-v$). Then integrating over the characteristic line $S_t$,
	\begin{align*}
	h(s + t,x + v t, v) = h(s, x, v) + \int_{s}^{t}\tilde{H}(s + \tau, x + \tau v, v)\bd\tau, x \in \Omega.
	\end{align*}
	We just need to pay careful attention to the particles which locate  at the boundary or are going to hit the boundary. When the molecule hits the boundary at $x_{1}$ with velocity $v$, the velocity after collision is given by the specular reflection formula, $R_{x_{1}} v = v - 2[n(x_{1}) v] n(x_{1})$.
	Then we can set the point $(t_1, x_{t_1}, R_{x_{1}}v)$ as new "initial" point. Similarly, we can define $t_{2}$, the maximal forward time of staying in $\mathcal{O}$(with velocity $R_{x_{t_1}} v$) starting from $x_1$ and the trajectory
	\[ S_t(s, x, v) = (s + t, x_{1} + R_{x_{t_1}} v (t - t_1), R_{x_{t_1}} v),~~ S_t \in \mathcal{O},~ t_1 \le t < t_2. \]
	According to the boundary condition,
	\[ \gamma_- h(t_1, x_1, R_{x_{t_1}} v) = z(t_1, x_1, R_{x_{t_1}} v), \]
	then for the  particle located  $ x_{1} \in \partial\Omega$ with velocity $R_{x_{1}} v$ and $t_1 \le t < t_2$,
	\begin{align*}
	h(t,x_{1} + R_{x_1} v (t-t_1) , R_{x_1} v) = z(t_1, x_{1} , R_{x_1} v) + \int_{t_1}^{t}\tilde{H}( \tau,x_{1} + R_{x_1} v (\tau - t_1) , R_{x_1} v)\bd\tau.
	\end{align*}
Inductively,  we can define its trajectory  $S_t(s,x,v)$ for any $t >0$ and solve \eqref{boltzmann-approximate-linear} by integrating over the characteristic line $S_t$.
	
	We start to deduce some estimates. Let $\beta_\delta$ be a sequence of even smooth functions,  such that
	\[ \beta_\delta(0)=0, \beta_\delta(y) \ge 0,  ~ |\beta_\delta'(y)| \le 1, \beta_\delta \to |y| ( \delta \to 0).   \]
	By Lemma \ref{green-formula-Lp}, multiplying the first equation in \ref{boltzmann-approximate-linear}, we can infer that
	\begin{align}
	\label{estimate-of-linear-linear}
	\begin{split}
	& \int_{\mathcal{O}} \beta_\delta(h)(t) \mathrm{d}v \mathrm{d}x + \int_{0}^{t}\int_{\Sigma_+} \beta_\delta(\gamma_+ h) \mathrm{d}\mu \mathrm{d}\sigma_x\mathrm{d}s\qquad\qquad  \\
	& \le \int_{0}^{t}\int_{\Sigma_-} \beta_\delta(z) \mathrm{d}\mu \mathrm{d}\sigma_x\mathrm{d}s + \int_{{\mathcal{O}_T}} |\tilde{H}| \mathrm{d}v \mathrm{d}x \bd s + \int_{\mathcal{O}} \beta_\delta(F^n_0) \mathrm{d}v \mathrm{d}x\\
	& \le \int_{0}^{t}\int_{\Sigma_-} |z| \mathrm{d}\mu \mathrm{d}\sigma_x\mathrm{d}s + \int_{{\mathcal{O}_T}} |\tilde{H}| \mathrm{d}v \mathrm{d}x \bd s + \int_{\mathcal{O}} |F^n_0| \mathrm{d}v \mathrm{d}x.
	\end{split}
	\end{align}
	We complete the proof by letting $\delta \to 0$.
	
\end{proof}
For   every $k\ge 0$ and any $T>0$, if we set $F^{n,k}(0)= F^n_0$,  by Theorem \ref{theorem-boltzmann-approximate-linear},  the following system
\begin{align*}
	\begin{cases}
	\partial_t F^{n,k+1} + v \cdot \nabla_x F^{n,k+1} = \Q^n(F^{n,k}, F^{n,k}),  v\in \mathbb{R}^3, \\
	F^{n,k}(0,x, v) = F_0^n(x, v),\\
	\gamma_-F^{n,k} =Z,~~\text{on}~~\Sigma_-,
	\end{cases}
	\end{align*}
	admits a unique solution $F^{n,k} \in L^\infty([0,T], L^1(\mathcal{O}))$.
	
In fact, for any fixed $n$,  as long as the life span is small enough, $F^{n,k}$ is a compact sequence in $L^1(\mathcal{O})$.
\begin{theorem}[Local-in-time existence to \eqref{boltzmann-approximate-lions}]
	\label{theorem-boltzmann-approximate-lions-local} For every $n$, under the assumptions on initial data \eqref{app-f} and on the incoming boundary condition \eqref{app-g}, there exists some $T_n$ such system \eqref{boltzmann-approximate-lions} admits a unique solution $F^n \in L^\infty([0,T_n];L^1(\mathcal{O}))$ in the sense of distribution. Further, there exists  a unique trace  $\gamma_+ F^n \in L^1((0,T_n)\times\Sigma_+ ) $    such that
	\begin{align*}
	&\int_{0}^{T_n} \int_{\mathcal{O}} \big( \beta(F^n)(\partial_t \phi  + v \cdot \nabla_x \phi) + \Q^n(F^n,F^n) \beta'(F^n)\phi\big) \mathrm{d}v \mathrm{d}x \mathrm{d}t \\
	& =   \int_{\mathcal{O}}  \phi(T_n) \beta(F^n)(T_n) \mathrm{d}v \mathrm{d}x   -    \int_{\mathcal{O}}  \phi(0) \beta(F^n_0)  \mathrm{d}v \mathrm{d}x   \\
	&  + \int_{0}^{T_n}  \int_{\Sigma_+} \phi \beta(\gamma_+ F^n) \bd \mu \bd \sigma_x  \bd t - \int_{0}^{T_n}  \int_{\Sigma_-} \beta( Z) \phi  \bd \mu \bd \sigma_x  \bd t,
	\end{align*}
	for  all $\beta(\cdot ) \in W^{1,\infty}_{loc}(\mathbb{R})$ with $\sup_{x \in \mathbb{R}}|\beta'(x)| < \infty$   and all the test function $\phi \in \mathcal{D}({[0,T]\times\bar{\Omega}\times\mathbb{R}^3})$.
	Moreover, there exists a constant $C_n$ such that the solution and its trace satisfy the following estimates
\begin{align}
\label{est-theorem-approximate-lions-local}
\begin{split}
\int_{\mathcal{O}} |F^n|(t) \mathrm{d}v \mathrm{d}x + \int_{0}^{t}\int_{\Sigma_+} |\gamma_+ F^n| \mathrm{d}\mu \mathrm{d}\sigma_x\mathrm{d}s \\
\le C_n \bigg( \int_{\mathcal{O}} |F_0^n| \mathrm{d}v \mathrm{d}x + \int_{0}^{t}\int_{\Sigma_-} |Z| \mathrm{d}\mu \mathrm{d}\sigma_x\mathrm{d}s \bigg), ~ t\le T_n.
\end{split}
\end{align}
\end{theorem}
\begin{proof}

	We are going to use iteration methods to prove this theorem. For each $n\ge 1$, we can obtain  the existence of $F^{n,1}$ by applying Theorem \ref{theorem-boltzmann-approximate-linear} to the following system:
	\begin{align}
		\label{boltzmann-approximate-lions-ite-1}
		\begin{cases}
			\partial_t F^{n,1} + v \cdot \nabla_x F^{n,1} = Q^n(F_0, F_0), \\
			F^{n,1}(0,x, v) = F_0(x, v),\\
			\gamma_-F^{n,1} =Z,~~\text{on}~~\Sigma_- .
		\end{cases}
	\end{align}

Recalling
\begin{align*}
\|\Q^n(F, F)\|_{L^1({\mathcal{O}})} \le C_n \|F\|_{L^1({\mathcal{O}})},
\end{align*} inductively, we can obtain a solution sequence   $\{F^{n,k}\}$ from the following iteration system for each $k\ge 1$
	\begin{align}
	\label{boltzmann-approximate-lions-ite-k}
	\begin{cases}
	\partial_t F^{n,k+1} + v \cdot \nabla_x F^{n,k+1} = Q^n(F^{n,k}, F^{n,k}), \\
	F^{n,k}(0,x, v) = F_0(x, v),\\
	\gamma_-F^{n,k} =Z,~~\text{on}~~\Sigma_- .
	\end{cases}
	\end{align}
	Moreover, $F^{n,k}$ satisfies
	\begin{align}
	\label{estimate-boltzmann-approximate-lions-ite-k}
	\sup\limits_{0 \le s \le t}\int_{\mathcal{O}} |F^{n,k}|(s) \mathrm{d}v \mathrm{d}x + \int_{0}^{t}\int_{\Sigma_+} |\gamma_+ F^{n,k}| \mathrm{d}\mu \mathrm{d}\sigma_x\mathrm{d}s\qquad\qquad \nonumber\\
	\le C_n\bigg( \int_{{\mathcal{O}_T}} |F^{n,k-1}| \mathrm{d}v \mathrm{d}x \bd s + \int_{\mathcal{O}} |F_0| \mathrm{d}v \mathrm{d}x + \int_{0}^{t}\int_{\Sigma_-} | Z| \mathrm{d}\mu \mathrm{d}\sigma_x\mathrm{d}s\bigg).
	\end{align}
	Denote:
	\[  C_{0,z,1} = \int_{\mathcal{O}} |F_0| \mathrm{d}v \mathrm{d}x + \int_{0}^{1}\int_{\Sigma_-} | Z| \mathrm{d}\mu \mathrm{d}\sigma_x\mathrm{d}s. \]
	Then, for $ t \le  1$,
		\begin{align}
		\label{estimate-boltzmann-approximate-lions-ite-uni-0}
        \begin{split}
		&\sup\limits_{0 \le s \le t}\int_{\mathcal{O}} |F^{n,k}|(s) \mathrm{d}v \mathrm{d}x + \int_{0}^{t}\int_{\Sigma_+} |\gamma_+ F^{n,k}| \mathrm{d}\mu \mathrm{d}\sigma_x\mathrm{d}s  \\
		&\le C_n\bigg( \int_{{\mathcal{O}_t}} |F^{n,k-1}| \mathrm{d}v \mathrm{d}x \bd s + C_{0,z,1}\bigg)\\
		&\le C_nt  \int_{{\mathcal{O}_t}} |F^{n,k-2}| \mathrm{d}v \mathrm{d}x \bd s + (1 + C_n t) C_{0,z,1}\\
		&\le (C_nt)^{k-1} t \int_{\mathcal{O}} |F_0| \mathrm{d}v \mathrm{d}x  + (1 + C_n t + \cdots + (C_nt)^{k-1}) C_{0,z,1}.
	    \end{split}
		\end{align}

	Choosing small enough $t^n$ such that $C_n t^n < 1$,  there exists $\tilde{C}_n$ such that for any $k\in \mathbb{N}^+$
			\begin{align}
			\label{estimate-boltzmann-approximate-lions-ite-uni}
			\begin{split}
			&	\sup\limits_{0 \le s \le t}\int_{\mathcal{O}} |F^{n,k}|(s) \mathrm{d}v \mathrm{d}x + \int_{0}^{t}\int_{\Sigma_+} |\gamma_+ F^{n,k}| \mathrm{d}\mu \mathrm{d}\sigma_x\mathrm{d}s \le \tilde{C}_n C_{0,z,1},~ t \le t^n.
			\end{split}
			\end{align}
			In fact, $\{F^{n,k}\}$ is a convergent sequence in $L^\infty([0,t^n], L^1(\mathcal{O}))$.  Noticing system \eqref{boltzmann-approximate-lions-ite-k} is a linear equation, we can infer
				\begin{align}
				\label{boltzmann-approximate-lions-ite-k-con}
				\begin{cases}
				\partial_t (F^{n,k+1} - F^{n,k})  + v \cdot \nabla_x (F^{n,k+1} - F^{n,k}) = Q^n(F^{n,k}, F^{n,k}) - Q^n(F^{n,k-1}, F^{n,k-1}),\\
				(F^{n,k+1} - F^{n,k})(0,x, v) = 0,\\
				\gamma_-(F^{n,k+1} - F^{n,k}) =0,~~\text{on}~~\Sigma_- .
				\end{cases}
				\end{align}
				By simple calculation,  one gets
				\begin{align}
				\label{estimate-boltzmann-approximate-lions-ite-con}
				&\sup\limits_{0 \le s \le t}\int_{\mathcal{O}} |F^{n,k+1} - F^{n,k}|(s) \mathrm{d}v \mathrm{d}x + \int_{0}^{t}\int_{\Sigma_+} |\gamma_+ (F^{n,k+1} - F^{n,k}) | \mathrm{d}\mu \mathrm{d}\sigma_x\mathrm{d}s\qquad\qquad \nonumber\\
				&\le \tilde{C}_n t \sup\limits_{0 \le s \le t}\int_{\mathcal{O}} |F^{n,k} - F^{n, k -1}|(s) \mathrm{d}v \mathrm{d}x,
				\end{align}
where we have use
\[  \|Q^n(F, F)-Q^n(G,G)\|_{L^1({\mathcal{O}})} \le C_n \|F-G\|_{L^1({\mathcal{O}})}. \]
Choosing $t^n_1 < t^n$ such that $t^n_1 \tilde{C}_n  < 1$, then $\{F^{n,k}\}$ is a convergent sequence. Resetting $T^n = t^n_1$, we complete the proof.
\end{proof}

The life span  $T^n$ in Theorem \ref{theorem-boltzmann-approximate-lions-local}    does not have a lower bound with $n$. This results from the source term. With the help of the symmetry properties of the Boltzmann collision kernel, we conclude

\begin{theorem}[global-in-time existence to \eqref{boltzmann-approximate-lions}]
	\label{theorem-boltzmann-approximate-lions-global} For any $T>0$, under the assumptions \eqref{app-g} and \eqref{app-f},  for every $n$,  system \eqref{boltzmann-approximate-lions} has a unique solution $F^n \in L^\infty([0,T];L^1(\mathcal{O}))$ such that
	\[ \partial_t F^n + v \cdot \nabla_x F^n = \Q^n(F^n, F^n) \]
	holds in the sense of distribution. Further, there exists  a unique trace $\gamma_+ F  \in L^1((0,T)\times\Sigma_+) $    such that
	\begin{align*}
	&\int_{0}^{T} \int_{\mathcal{O}} \big( \beta(F^n)(\partial_t \phi  + v \cdot \nabla_x \phi) + \Q^n(F^n,F^n) \beta'(F^n)\phi\big) \mathrm{d}v \mathrm{d}x \mathrm{d}t \\
	& =   \int_{\mathcal{O}}  \phi(T) \beta(F^n)(T) \mathrm{d}v \mathrm{d}x   -    \int_{\mathcal{O}}  \phi(0) \beta(F^n_0)  \mathrm{d}v \mathrm{d}x   \\
	&  + \int_{0}^{T}  \int_{\Sigma_+} \phi \beta(\gamma_+ F^n) \bd \mu \bd \sigma_x  \bd t - \int_{0}^{T}  \int_{\Sigma_-} \beta( Z) \phi  \bd \mu \bd \sigma_x  \bd t,
	\end{align*}
	for all $\beta(\cdot ) \in W^{1,\infty}_{loc}(\mathbb{R})$ with $\sup_{x \in \mathbb{R}}|\beta'(x)| < \infty$   and all the test function $\phi \in \mathcal{D}({[0,T]\times\bar{\Omega}\times\mathbb{R}^3})$.
 Furthermore, $F^n$ and $\gamma_+ F^n$ satisfy
 \begin{itemize}
 	\item local conservation law of mass: for any $\phi \in \mathcal{D}(\bar{\Omega})$,
 	\begin{align}
\label{est-theorem-approximate-lions-local-density}
\begin{split}
&\int_{\mathcal{O}}   F^n(t) \phi  \mathrm{d}v \mathrm{d}x + \int_{0}^{t}\int_{\Sigma_+}  \gamma_+ F^n  \phi \dd \mu \mathrm{d}\sigma_x\mathrm{d}s \\
&=  \int_{\mathcal{O}} F^n_0  \phi \mathrm{d}v \mathrm{d}x  +  \int_{0}^{t}\int_{\Sigma_-} Z \phi  \dd \mu \mathrm{d}\sigma_x\mathrm{d}s,\qquad t \le T.
\end{split}
\end{align}
 	\item global conservation law of mass:
 	\begin{align}
 	\label{est-theorem-approximate-lions-global-density}
 	\begin{split}
 	&\int_{\mathcal{O}}   F^n(t) \mathrm{d}v \mathrm{d}x + \int_{0}^{t}\int_{\Sigma_+}  \gamma_+ F^n \dd \mu \mathrm{d}\sigma_x\mathrm{d}s \\
 	&=  \int_{\mathcal{O}} F^n_0 \mathrm{d}v \mathrm{d}x  +  \int_{0}^{t}\int_{\Sigma_-} Z \dd \mu \mathrm{d}\sigma_x\mathrm{d}s,\qquad t \le T.
 	\end{split}
 	\end{align}
 	\item global conservation law of momentum
 	\begin{align}
 	\label{est-theorem-approximate-lions-global-mo}
 	\begin{split}
 	&\int_{\mathcal{O}}  v F^n(t) \mathrm{d}v \mathrm{d}x + \int_{0}^{t}\int_{\Sigma_+} v \gamma_+ F^n \dd \mu \mathrm{d}\sigma_x\mathrm{d}s \\
 	& =  \int_{\mathcal{O}} v F^n_0 \mathrm{d}v \mathrm{d}x  +  \int_{0}^{t}\int_{\Sigma_-} v Z \dd \mu \mathrm{d}\sigma_x\mathrm{d}s, \qquad t \le T.
 	\end{split}
 	\end{align}
 	\item global conservation law of energy
 	\begin{align}
 	\label{est-theorem-approximate-lions-global-energy}
 	\begin{split}
 	&\int_{\mathcal{O}}  | v|^2F^n(t) \mathrm{d}v \mathrm{d}x + \int_{0}^{t}\int_{\Sigma_+} | v|^2\gamma_+ F^n \dd \mu \mathrm{d}\sigma_x\mathrm{d}s \\
 	&\le  \int_{\mathcal{O}}  | v|^2F^n_0 \mathrm{d}v \mathrm{d}x  +  \int_{0}^{t}\int_{\Sigma_-}| v|^2 Z \dd \mu \mathrm{d}\sigma_x\mathrm{d}s, \qquad t \le T.
 	\end{split}
 	\end{align}
 	\item global entropy inequality
 	\begin{align}
 	&\label{est-theorem-app-lions-entropy}
 	\int_{\mathcal{O}}  F^n\log F^n(t) \mathrm{d}v \mathrm{d}x + \int_{0}^{t}\int_{\Sigma_+}\gamma_+F^n  \log \gamma_+ F^n \dd \mu \mathrm{d}\sigma_x\mathrm{d}s +  \int_{0}^t \mathcal{D}(F^n)(s)\mathrm{d}s  \nonumber \\
 	& \le \int_{\mathcal{O}}  F^n_0\log F^n_0 \mathrm{d}v \mathrm{d}x + \int_{0}^{t}\int_{\Sigma_-}Z \log Z \dd \mu \mathrm{d}\sigma_x\mathrm{d}s,\qquad  t \le T,
 	\end{align}
 	\item global relative entropy inequality
 	\begin{align}
 	&    H(F^n| \mathrm{M} )(t) + \int_{0}^{t}\int_{\Sigma_+}h(\gamma_+ F^n/\mathrm{M} ) \dd \mu \mathrm{d}\sigma_x\mathrm{d}s +  \int_{0}^t \mathcal{D}(F^n)(s)  \mathrm{d}s\nonumber\\
 	& \label{est-theorem-app-lions-global-relative-entropy} \le  H(F_0^n| \mathrm{M} )  + \int_{0}^{t}\int_{\Sigma_-}h(Z| \mathrm{M}) \dd \mu \mathrm{d}\sigma_x\mathrm{d}s,\qquad   t  \le T.
 	\end{align}
 \end{itemize}

\end{theorem}
\begin{proof}
Noticing that $1$ and $|v|^2$  lay in the kernel of $\mathcal{Q}^n$, we infer that
	\begin{align}
	\label{est-app-lions-density}
	\begin{split}
	& \int_{\mathcal{O}} F^n(t) \mathrm{d}v \mathrm{d}x + \int_{0}^{t}\int_{\Sigma_+}  \gamma_+ F^n \mathrm{d}\mu \mathrm{d}\sigma_x\mathrm{d}s \\
	& =    \int_{\mathcal{O}} F_0^n \mathrm{d}v \mathrm{d}x + \int_{0}^{t}\int_{\Sigma_-}  Z  \mathrm{d}\mu \mathrm{d}\sigma_x\mathrm{d}s , ~ t > 0,
	\end{split}
	\end{align}
and
		\begin{align}
		\label{est-app-lions-energy}
		\begin{split}
		& \int_{\mathcal{O}} |v|^2F^n(t) \mathrm{d}v \mathrm{d}x + \int_{0}^{t}\int_{\Sigma_+}  |v|^2 \gamma_+ F^n \mathrm{d}\mu \mathrm{d}\sigma_x\mathrm{d}s \\
		& =    \int_{\mathcal{O}} |v|^2F_0^n \mathrm{d}v \mathrm{d}x + \int_{0}^{t}\int_{\Sigma_-} |v|^2 Z  \mathrm{d}\mu \mathrm{d}\sigma_x\mathrm{d}s , ~ t > 0.
		\end{split}
		\end{align}
Then for any $t>0$, we find
	\begin{align}
	\label{est-app-lions-density+energy}
	\begin{split}
	& \sup\limits_{0 \le s \le t}\int_{\mathcal{O}} (1 + |v|^2)F^n(s) \mathrm{d}v \mathrm{d}x + \int_{0}^{t}\int_{\Sigma_+} (1 + |v|^2) \gamma_+ F^n \mathrm{d}\mu \mathrm{d}\sigma_x\mathrm{d}s \\
	& \le    \int_{\mathcal{O}} (1 + |v|^2) F_0^n \mathrm{d}v \mathrm{d}x + \int_{0}^{t}\int_{\Sigma_-} (1 + |v|^2) Z  \mathrm{d}\mu \mathrm{d}\sigma_x\mathrm{d}s.
	\end{split}
	\end{align}
	Using the same trick as the one in \cite[pp.360]{diperna-lions1989cauchy}, one finds that
	\begin{align}
	\label{estimate-interior-positive}
	F^n(t,x,v)>0,~~  t \le T^n, (x,v) \in \mathcal{O}.
	\end{align}
	Moreover,
	\begin{align}
	\label{estimate-boundary-positive}
	\gamma_+ F^n(t,x,v)>0,~~  t \le T^n, (x,v) \in \Sigma_+.
	\end{align}		
For the relative entropy inequality, multiplying the first equation of \eqref{boltzmann-approximate-lions} by $ 1 + \log F$, $\log \mathrm{M} $ respectively,  we have
\begin{align}
\label{boltzmann-approximate-lions-entropy}
\partial_t (F^n\log F^n) + v \cdot \nabla_x (F^n\log F^n) = Q^n(F^n, F^n)(1+\log F^n),\\
\label{boltzmann-approximate-lions-r-entropy-2}
\partial_t (F^n\log \mathrm{M} ) + v \cdot \nabla_x (F^n\log \mathrm{M} ) = Q^n(F^n, F^n)( \log \mathrm{M} ).
\end{align}
In the light of \eqref{boltzmann-approximate-lions-entropy} and \eqref{boltzmann-approximate-lions-r-entropy-2},  we can infer that
\begin{align}
\label{boltzmann-approximate-lions-r-entropy}
\partial_t  H(F^n| \mathrm{M} ) + v \cdot \nabla_x  H(F^n| \mathrm{M} ) = Q^n(F^n, F^n)\log f_{\mathrm{M}}^n.
\end{align}
For any $t>0$, integrating \eqref{boltzmann-approximate-lions-r-entropy} over $\mathcal{O}_t$, we conclude that
\begin{align*}
& \sup\limits_{0\le s \le t}  \int_{\mathcal{O}}H(F^n| \mathrm{M} )(s)\dd v \dd x + \int_{0}^{t}\int_{\Sigma_+}\gamma_+ \big(H(F^n|\mathrm{M} )\big)(s) \mathrm{d}\mu \mathrm{d}\sigma_x\mathrm{d}s +  \int_{0}^t \mathcal{D}(F^n)(s) \mathrm{d}s\\
& \le \int_{\mathcal{O}}H(F_0^n| \mathrm{M} )(s)\dd v \dd x + \int_{0}^{t}\int_{\Sigma_-} H(Z| \mathrm{M} ) \mathrm{d}\mu \mathrm{d}\sigma_x\mathrm{d}s.
\end{align*}
Denoting
\[ C_{H,z,t} = \int_{\mathcal{O}}H(F_0^n| \mathrm{M})(t)\dd v \dd x + \int_{0}^{t}\int_{\Sigma_-} H(Z| \mathrm{M} ) \mathrm{d}\mu \mathrm{d}\sigma_x\mathrm{d}s, \]
and recalling that relative entropy is always positive, the above inequality indicates  that
\begin{align}
&\label{est-bolt-app-lions-re-entropy-interior}
\int_{\mathcal{O}}H(F^n| \mathrm{M} )(t)\dd v \dd x  \le  C_{H,z,t},\\
&\label{est-bolt-app-lions-re-entropy-boundary}
\int_{0}^{t}\int_{\Sigma_+}\gamma_+ \bigg(H(F^n| \mathrm{M} )\bigg)(s) \mathrm{d}\mu \mathrm{d}\sigma_x\mathrm{d}s \le C_{H,z,t}.
\end{align}
Similarly,  for any $t>0$,  integrating \eqref{boltzmann-approximate-lions-entropy} over $\mathcal{O}_t$, we have
	\begin{align}
	\label{est-boltz-app-lions-entropy}
	\begin{split}
	& \sup\limits_{0\le s \le t}\int_{\mathcal{O}}  F^n\log F^n(s) \mathrm{d}v \mathrm{d}x + \int_{0}^{t}\int_{\Sigma_+}\gamma_+ (F^n \log F^n)(s) \mathrm{d}\mu \mathrm{d}\sigma_x\mathrm{d}s\\
	& + \int_{0}^{t}  \mathcal{D}(F^n)(s) \mathrm{d}s
	\le     \int_{\mathcal{O}} F_0^n\log F^n_0 \mathrm{d}v \mathrm{d}x + \int_{0}^{t}\int_{\Sigma_-}  Z\log Z(s) \mathrm{d}\mu \mathrm{d}\sigma_x\mathrm{d}s  .
	\end{split}
	\end{align}
With the help of  \eqref{est-app-lions-density+energy} and \eqref{est-boltz-app-lions-entropy}, by simple calculation, we find that
	\begin{align}
	\label{est-boltz-app-lions-entropy-abs}
	\begin{split}
	 \sup\limits_{0\le s \le t}\int_{\mathcal{O}}  F^n|\log F^n|(s) \mathrm{d}v \mathrm{d}x + \int_{0}^{t}\int_{\Sigma_+}\gamma_+F^n |\gamma_+ \log F^n| \mathrm{d}\mu \mathrm{d}\sigma_x\mathrm{d}s  \le C(t).
	\end{split}
	\end{align}		
Thus, the life span  $T^n$   in Theorem \ref{theorem-boltzmann-approximate-lions-local} can be extended  to any $T>0$.

\end{proof}

\section{weak compactness and global existence}
In this section, we  prove Theorem \ref{main-result-cut}. First,  we summarize all these estimates on $F^n$ up
\begin{align}
\label{conservations-cut}
\sup_{0\le s \le T} \int F^n(s)(1   + |v|^2 + |\log F^n|) \mathrm{d}v \mathrm{d}x  \le C(T),
\end{align}
and
\begin{align}
\label{conservation-boundary-cut}
\int_{0}^{t}\int_{\Sigma_\pm} (1 + |v|^2 + |\log \gamma_\pm F^n|)\gamma_\pm F^n \mathrm{d}\mu \mathrm{d}\sigma_x \mathrm{d}s \le C(T).
\end{align}
From estimates \eqref{conservations-cut}, for any fixed $T>0$, using Dunford-Pettis Lemma, it follows that $\{F^n\}$ and $\{\gamma_+ F^n\}$ are weakly compact sequence in $L^\infty((0,T);L^1(\mathcal{O}))$ and $L^1(0,T)\times\Sigma_+)$. Further, taking the whole sequence as example, there exists some $f \in L^1 ((0,T) \times \mathcal{O})$ such that
\begin{align}
\label{weak-convergence-f-cut}
F^n \rightharpoonup F,~~\text{in}~~L^1 ((0,T) \times \mathcal{O}).
\end{align}

The proof of Theorem \ref{main-result-cut} is split into two parts: the interior domain part and the boundary part.
\subsection{Interior domain}

\begin{theorem}[Extended Stability \cite{diperna-lions1989cauchy,lions1994compact}]
	\label{theorem-interior-cut}
	Let $B$ in \eqref{collision-original} satisfy the assumptions \eqref{Grad-cutoff-1} and \eqref{Grad-cutoff-2}.  Let $\{F^n\}$ be a sequence of   solutions obtained in Theorem \ref{theorem-boltzmann-approximate-lions-global} to the  approximate Boltzmann equation \eqref{boltzmann-approximate-lions}. Then
	\begin{itemize}
		\item $F^n \to F$ in  $L^p([0,T], L^1(\mathcal{O}))$, $p \ge 1$;
		\item for all nonlinearity $\beta \in C^1(\mathbb{R}_+, \mathbb{R}_+)$ satisfying
		\[   0 \le   \beta'(F)   < C { (1 +f)^{-1}},    \]	
		then
		\[   \beta'(F) \Q(F,F) \in L^1 ((0,T)\times\mathcal{O}),   \]
		and
		\[   \partial_t \beta(F) + v \cdot \nabla_x \beta(F) = \beta'(F)\Q(F,F),   \]
		holds in the sense of distribution.
		\item for all $\phi \in \mathcal{D}([0,T]\times\bar{\Omega}\times\mathbb{R}^3)$,
		\begin{align*}
		\int_{0}^{t}\int_\mathcal{O} \Q^n(F^n,F^n)\beta'(F^n) \phi \bd v \bd x \bd s \to \int_{0}^{t}\int_\mathcal{O} \Q(F,F)\beta'(F) \phi \bd v \bd x \bd s,\qquad t \le T.
		\end{align*}
	\end{itemize}
\end{theorem}
\begin{proof}
	The first two items directly come from \cite{diperna-lions1989cauchy,lions1994compact}. With the strong convergence of $\{F^n\}$ in $L^1$ space, the third can be verified by the argument in \cite{diperna-lions1989cauchy} too.
\end{proof}

\subsection{Boundary Parts} From Theorem \ref{theorem-interior-cut}, the solution $F^n$ satisfies
		\[   \partial_t \beta(F^n) + v \cdot \nabla_x \beta(F^n) = \beta'(F^n)\Q(F^n,F^n),~~\text{in}~~ \mathcal{D}'((0,T)\times\mathcal{O}).    \]
By Lemma \ref{green-formula-Lp}, we can only get the trace of $\beta(F)$ other than that of $f$. Recall
	\begin{align}
\label{est-app-lions-density-cp}
\begin{split}
& \int_{\mathcal{O}} F^n(t) \mathrm{d}v \mathrm{d}x + \int_{0}^{t}\int_{\Sigma_+}   \gamma_+ F^n \mathrm{d}\mu \mathrm{d}\sigma_x\mathrm{d}s \\
& =    \int_{\mathcal{O}}  F_0 \mathrm{d}v \mathrm{d}x + \int_{0}^{t}\int_{\Sigma_-}   Z \mathrm{d}\mu \mathrm{d}\sigma_x\mathrm{d}s  , ~ t > 0.
\end{split}
\end{align}
Since $F^n \to F$ in  $L^\infty([0,T], L^1(\Omega\times\mathbb{R}^3))$,
we get
\begin{align*}
\int_{\mathcal{O}} F^n(s) \mathrm{d}v \mathrm{d}x \to \int_{\mathcal{O}} F(s) \mathrm{d}v \mathrm{d}x,~~ s \le T.
\end{align*}
Then according to \eqref{conservations-cut}, by Dunford-Pettis Lemma,   we can infer: There exists   $f_\gamma \in L^1((0,T)\times\Sigma_+)$ such that
\begin{align*}
& \gamma_+ F^n \rightharpoonup F_\gamma,~~\text{in}~~ L^1((0,T) \times \Sigma_+),\\
& \int_{0}^{t}\int_{\Sigma_+}   \gamma_+ F^n \mathrm{d}\mu \mathrm{d}\sigma_x\mathrm{d}s \to  \int_{0}^{t}\int_{\Sigma_+}    F_\gamma \mathrm{d}\mu \mathrm{d}\sigma_x\mathrm{d}s.
\end{align*}
Thus,  we infer that
\begin{align}
\label{est-lions-density-0}
\begin{split}
& \int_{\mathcal{O}} F \mathrm{d}v \mathrm{d}x + \int_{0}^{t}\int_{\Sigma_+}   F_\gamma \mathrm{d}\mu \mathrm{d}\sigma_x\mathrm{d}s \\
& =    \int_{\mathcal{O}}  F_0 \mathrm{d}v \mathrm{d}x + \int_{0}^{t}\int_{\Sigma_-}   Z \mathrm{d}\mu \mathrm{d}\sigma_x\mathrm{d}s , ~ t > 0.
\end{split}
\end{align}
 	In fact, we will show that $F_\gamma$ satisfies \eqref{est-definition-trace-of-renormalized-cutoff}, namely that $F_\gamma$ is the trace of $F$ on $\Sigma_+$, $\gamma_+ F  : = F_\gamma$.

\begin{theorem}
	\label{theorem-boundary}
	Assume that $\beta \in C^1[0,+\infty)$ be a  non-linear function with $\beta'(x) \ge 0$ and $\sup\limits_{x\ge 0}(1+x) \beta'(x) < \infty $. Then for  any $\phi \in \mathcal{D}([0,T]\times\bar{\Omega}\times\mathbb{R}^3)$ and any $T>0$,
	\begin{align*}
	&\int_{0}^{T} \int_{\mathcal{O}} \big( \beta(F)(\partial_t \phi  + v \cdot \nabla_x \phi) +   \Q(f, f)\beta'(F) \phi\big) \mathrm{d}v \mathrm{d}x \mathrm{d}t \\
	& =   \int_{\mathcal{O}} \beta(F)(T) \phi  \mathrm{d}v \mathrm{d}x - \int_{\mathcal{O}} \beta(F_0) \phi(0)  \mathrm{d}v \mathrm{d}x  \\
	& + \int_{0}^{T}  \int_{\Sigma_+} \beta(F_\gamma) \phi \mathrm{d}\mu \mathrm{d}\sigma_x \mathrm{d}t - \int_{0}^{T}  \int_{\Sigma_-} \beta(Z) \phi \mathrm{d}\mu \mathrm{d}\sigma_x \mathrm{d}t.
	\end{align*}	
	
\end{theorem}
\begin{proof}
	
According to Theorem \ref{theorem-boltzmann-approximate-lions-global}, for all $T>0$, as long as $0 \le \beta'(x) \le C(1+x)^{-1}$, then for $\phi \in \mathcal{D}([0,T]\times\bar{\Omega}\times\mathbb{R}^3)$,
	\begin{align*}
	&\int_{0}^{T} \int_{\mathcal{O}} \big( \beta(F^n)(\partial_t \phi  + v \cdot \nabla_x \phi) + \Q^n(F^n,F^n) \beta'(F^n)\phi\big) \mathrm{d}v \mathrm{d}x \mathrm{d}t \\
	& =   \int_{\mathcal{O}}  \phi(T) \beta(F^n)(T) \mathrm{d}v \mathrm{d}x   -    \int_{\mathcal{O}}  \phi(0) \beta(F^n_0)  \mathrm{d}v \mathrm{d}x   \\
	&  + \int_{0}^{T}  \int_{\Sigma_+} \phi \beta(\gamma_+ F^n) \bd \mu \bd \sigma_x  \bd t - \int_{0}^{T}  \int_{\Sigma_-} \beta(Z) \phi  \bd \mu \bd \sigma_x  \bd t,
	\end{align*}

	Choose a sequence of concave function $\beta_j(x) = \frac{x}{ 1 + \frac{x}{j}}$, obviously,
	\[  \beta_j(f) \le j, ~\beta_n(0) =0,~~ 0<\beta_j'= \tfrac{j^2}{ (j + {f})^2}.    \]
	Besides,
	\[ 0 \le \beta_j(\gamma_+ F^n ) \le \gamma_+ F^n,  \]
	since $\{\gamma_+ F^n\}$ is a weakly compact sequence, for any fixed $j \in \mathbb{N}^+$,   there exists  $F_{\gamma,j} \in L^1((0,T)\times\Sigma_+)$ such that
	\[ \beta_j(\gamma_+ F^n) \rightharpoonup F_{\gamma, j}, ~~\text{in}~~L^1((0,T)\times\Sigma_+).  \]	
	Let $n$ go to infinity, with the help of Theorem \ref{theorem-interior-cut},
	\begin{align*}
	&\int_{0}^{T} \int_{\mathcal{O}} \big( \beta_j(F)(\partial_t \phi  + v \cdot \nabla_x \phi) + \Q(F,F) \beta'_j(F)\phi\big) \mathrm{d}v \mathrm{d}x \mathrm{d}t \\
	& =   \int_{\mathcal{O}}  \phi(T) \beta_j(f)(T) \mathrm{d}v \mathrm{d}x   -    \int_{\mathcal{O}}  \phi(0) \beta_j(F_0)  \mathrm{d}v \mathrm{d}x   \\
	&  + \int_{0}^{T}  \int_{\Sigma_+} \phi f_{\gamma,j} \bd \mu \bd \sigma_x  \bd t - \int_{0}^{T}  \int_{\Sigma_-} \beta_j(Z) \phi  \bd \mu \bd \sigma_x  \bd t,
	\end{align*}
	Noticing that $\beta_j(x) \le j$, then by Lemma \ref{green-formula-L-infty-defect}, $\gamma_+ \beta_j(f) = f_{\gamma,j}$ and
	\begin{align*}
	F_{\gamma, j + 1} \ge F_{\gamma, j},~~ a.e.  ~~ \text{on}~~ (0,T)\times\Sigma_+.
	\end{align*}
	Furthermore,  for all $\beta \in C^1(\mathbb{R}_+)$ with $0 \le \beta'(x)$ and $ \sup_{x \ge 0}(1 +x)\beta'(x) < \infty$,
	\begin{align*}
	&\int_{0}^{T} \int_{\mathcal{O}} \big( \beta(\beta_j (F))(\partial_t \phi  + v \cdot \nabla_x \phi) + \Q(F,F) \beta'_j(F)\beta'(\beta_j(f))\phi\big) \mathrm{d}v \mathrm{d}x \mathrm{d}t \\
	& =   \int_{\mathcal{O}}  \phi(T) \beta(\beta_j(F))(T) \mathrm{d}v \mathrm{d}x   -    \int_{\mathcal{O}}  \phi(0) \beta(\beta_j(F_0))  \mathrm{d}v \mathrm{d}x   \\
	&  + \int_{0}^{T}  \int_{\Sigma_+} \phi \beta(F_{\gamma,j}) \bd \mu \bd \sigma_x  \bd t - \int_{0}^{T}  \int_{\Sigma_-} \beta(\beta_j(Z)) \phi  \bd \mu \bd \sigma_x  \bd t,
	\end{align*}
	
	Recalling for any $j \in \mathbb{N}^+$, $ \beta(\beta_j(F))  \le    \beta(F)$,  by Lebesgue dominated convergence theorem,
	\[   \beta(\beta_j(F))  \to   \beta(F), ~~\text{in}~~  L^1((0,T)\times\mathcal{O}; \mathrm{d}v \mathrm{d}x \mathrm{d}t), \]
	and
	\[   \beta(\beta_j(F)) \to  \beta(F),  ~~\text{in}~~  L^\infty((0,T); L^1(\mathcal{O}; \mathrm{d} x\mathrm{d} v)). \]
	For the collision term, recalling that there exist a constant such that
	\[ |\beta'(\beta_j(F))| \le C (1 + \beta_j(F))^{-1},   \]
	then
	\begin{align*}
	|\Q(F,F) \beta'_j(F)\beta'(\beta_j(F))| & =  |\Q(F,F)\big( 1 + \frac{F}{n} \big)^{-2}\beta'(\beta_j(F))|\\
	& \le  C | \frac{\Q(F,F)}{1 + F} \frac{1 + F}{(1 + \frac{F}{n})^2} \times \frac{1}{ 1 + \tfrac{F}{1 + \tfrac{F}{n}} }\\
	& \le C | \frac{\Q(F,F)}{1 + F} \frac{1 + F}{(1 + \frac{F}{n})^2 + f( 1 + \frac{F}{n} )}|.  \\
	& \le C | \frac{\Q(F,F)}{1 + F}|.
	\end{align*}
	It follows that
	\[  \Q(F,F) \beta'_j(F)\beta'(\beta_j(F)) \to  \Q(F,F)\beta'(F),~~ a.e. \text{on}~~{\mathcal{O}_T},  ~~ \text{as}~~ j \to \infty.\]
	Thus, by  Lebesgue dominated convergence theorem again
	\begin{align*}
	\int_{0}^{T} \int_{\mathcal{O}}  \Q(F,F) \beta'_j(F)\beta'(\beta_j(F))\phi  \mathrm{d}v \mathrm{d}x \mathrm{d}t \to \int_{0}^{T} \int_{\mathcal{O}}  \Q(F,F)  \beta'(F)\phi  \mathrm{d}v \mathrm{d}x \mathrm{d}t.
	\end{align*}
	According to the definition   \eqref{def-r-convergence} and recalling
\[  \beta_j(F) \le j, ~\beta_n(0) =0,~~ 0<\beta_j'= \frac{j^2}{ (j + {F})^2},   \]	
	based on the above analysis, we can infer that $\{\beta_j\}$ is a renormalizing sequence and $\{F_{\gamma, j}\}$ is a increasing sequence on $(0,T)\times\Sigma_+$, together with $ \gamma_+ F^n \rightharpoonup f_\gamma$($L^1((0,T)\times\Sigma_+)$), by virtual of Proposition \ref{proposition-r-convergence},  we conclude
	\begin{align*}
	F_{\gamma, j} \uparrow F_\gamma,~~ \text{on}~~ (0,T)\times\Sigma_+.
	\end{align*}
	Let $j \to + \infty$, finally, we obtain
	\begin{align*}
	&\int_{0}^{T} \int_{\mathcal{O}} \big( \beta(F)(\partial_t \phi  + v \cdot \nabla_x \phi) +   \Q(F, F)\beta'(F) \phi\big) \mathrm{d}v \mathrm{d}x \mathrm{d}t \\
	& =   \int_{\mathcal{O}} \beta(F)(T) \phi  \mathrm{d}v \mathrm{d}x - \int_{\mathcal{O}} \beta(F_0) \phi(0)  \mathrm{d}v \mathrm{d}x  \\
	& + \int_{0}^{T}  \int_{\Sigma_+} \beta(F_\gamma) \phi \mathrm{d}\mu \mathrm{d}\sigma_x \mathrm{d}t - \int_{0}^{T}  \int_{\Sigma_-} \beta(z) \phi \mathrm{d}\mu \mathrm{d}\sigma_x \mathrm{d}t.
	\end{align*}
	This means
	\begin{align}
	\label{concave-equal}
	\gamma_+ F  = F_\gamma, ~~ \gamma_- F =Z.
	\end{align}
	\end{proof}
  Then \eqref{est-lions-density-0} becomes
\begin{align}
\label{est-lions-density}
\begin{split}
& \int_{\mathcal{O}} F \mathrm{d}v \mathrm{d}x + \int_{0}^{t}\int_{\Sigma_+}   \gamma_+ F  \xmus \\
& =    \int_{\mathcal{O}}  F_0 \mathrm{d}v \mathrm{d}x + \int_{0}^{t}\int_{\Sigma_-}   Z \xmus , ~ t > 0.
\end{split}
\end{align}
As for the energy inequality, recalling that
\[ \gamma_+ F^n \rightharpoonup \gamma_+ F, ~\text{in}~ L^1((0,T)\times\Sigma_+),  \]
then  for any fixed $m\in \mathbb{N}^+$, denoting by $\mathbf{1}_m$ the characteristic function of ball in $\mathbb{R}^3$ with radius $m$,  $\{v : |v| \le m  \}$, we can infer that
\[ |v|^2 \mathbf{1}_{m}\gamma_+ F^n \rightharpoonup  |v|^2 \mathbf{1}_{m} \gamma_+ F, ~\text{in}~ L^1((0,T)\times\Sigma_+),  \]
and
\[ |v|^2 \mathbf{1}_{m}  F^n \rightharpoonup  |v|^2 \mathbf{1}_{m}   F, ~\text{in}~ L^\infty((0,T); L^1(\mathcal{O})).  \]
By the lower semi-continuity of norm,
\begin{align*}
\begin{split}
& \sup_{0\le s \le t}\int_{\mathcal{O}} \mathbf{1}_{m} |v|^2 F(s) \mathrm{d}v \mathrm{d}x + \int_{0}^{t}\int_{\Sigma_+} \mathbf{1}_{m}  |v|^2 \gamma_+ F  \dd \mu \mathrm{d}\sigma_x\mathrm{d}s \\
& \le    \int_{\mathcal{O}} |v|^2F_0 \mathrm{d}v \mathrm{d}x + \int_{0}^{t}\int_{\Sigma_-} |v|^2 Z \dd \mu \mathrm{d}\sigma_x\mathrm{d}s, \qquad t > 0.
\end{split}
\end{align*}
Taking $m$ to infinity, by Fatou lemma,  we deduce
\begin{align}
\label{est-lions-energy}
\begin{split}
& \sup_{0\le s \le t}\int_{\mathcal{O}} |v|^2f(s) \mathrm{d}v \mathrm{d}x + \int_{0}^{t}\int_{\Sigma_+}  |v|^2 \gamma_+ F  \dd \mu \mathrm{d}\sigma_x\mathrm{d}s \\
& \le    \int_{\mathcal{O}} |v|^2F_0 \mathrm{d}v \mathrm{d}x + \int_{0}^{t}\int_{\Sigma_-} |v|^2 Z \xmus, \qquad t > 0.
\end{split}
\end{align}
For the relative entropy inequality, noticing that $h(z)$ is a positive convex function, by the lower semi-continuity of convex functions with respect to weak convergence, we deduce that
\begin{align}
\label{est-lions-entropy}
\begin{split}
& H(F| \mathrm{M} ) + \int_{0}^{t}\int_{\Sigma_+}h(\gamma_+ F/\mathrm{M} ) \dd \mu \mathrm{d}\sigma_x\mathrm{d}s +  \int_{0}^t \mathcal{D}(F)(s) \mathrm{d}s  \\
&\le  H(F_0| \mathrm{M} )  + \int_{0}^{t}\int_{\Sigma_-}h(Z| \mathrm{M} ) \dd \mu \mathrm{d}\sigma_x\mathrm{d}s,\qquad t \ge 0.
\end{split}
\end{align}

\subsection{Local conservation laws}
In this subsection, we focus on the  local conservation laws.
\subsubsection{Local   conservation law of mass}
Choosing some function $\phi_1 \in \mathcal{D}([0,T]\times\bar{\Omega})$, multiplying the first equation \eqref{boltzmann-approximate-lions} by $\phi_1$, integrating by parts, then we have
\begin{align*}
\begin{split}
& \int_\mathcal{O} F^n(T,x,v) \phi_1(T,x,v) \bd v \bd x - \int_\mathcal{O} F^n(0,x,v) \phi_1(0,x) \bd v \bd x  - \int_{0}^{T} \int_\mathcal{O} F^n(s,x,v) \partial_t \phi_1(s,x) \bd v \bd x \bd s\\
& = \int_{0}^{T} \int_\mathcal{O} F^n(s,x,v) v \cdot \nabla \phi_1(s,x) \bd v \bd x \bd s - \int_{0}^{T} \int_\Sigma \gamma F^n(s,x,v)   \phi_1(s,x)   \mathrm{n}(x) \cdot v \bd v \bd\sigma_x \bd s.
\end{split}
\end{align*}
As
\[ (1+ |v|) F^n \to (1 + |v|) F,~~\text{in}~~ L^1\big((0,T) \times \Omega\big), \]
and
\[ \gamma_\pm F^n \weak \gamma_\pm F,~~\text{in}~~ L^1\big((0,T) \times \Sigma_\pm\big), \]
taking $n\to \infty$, we have
\begin{align*}
\begin{split}
& \int_\mathcal{O} F(T,x,v) \phi_1(T,x) \bd v \bd x - \int_\mathcal{O} f(0,x,v) \phi_1(0,x) \bd v \bd x  - \int_{0}^{T} \int_\mathcal{O} f(s,x,v) \partial_t \phi_1(s,x) \bd v \bd x \bd s\\
& = \int_{0}^{T} \int_\mathcal{O} F(s,x,v) v \cdot \nabla \phi_1(s,x) \bd v \bd x \bd s - \int_{0}^{T} \int_{\Sigma_\pm} \gamma F(s,x,v)   \phi_1(s,x)  \mathrm{n}(x) \cdot v \bd v \bd\sigma_x \bd s.
\end{split}
\end{align*}
Noticing that $\phi_1$ is independent of $v$, then it can be rewritten as
\begin{align}
\label{commutative}
\begin{split}
& \int_\Omega \phi_1(T,x) \cdot \big( \int_{\mathbb{R}^3}F \bd v \big)(T,x)  \bd x - \int_\Omega \phi_1(0,x) \cdot \big( \int_{\mathbb{R}^3}F \bd v \big)(0,x)  \bd x  \\
& - \int_{0}^{T} \int_{\partial\Omega} \partial_t \phi_1(s,x) \cdot \big( \int_{\mathbb{R}^3} F\dd v \big)(s,x) \bd x \bd s\\
& = \int_{0}^{T} \int_{\Omega} \nabla \phi_1(s,x) \cdot  \big( \int_{\mathbb{R}^3} F v \bd v\big)(s,x) \bd x \bd s - \int_{0}^{T} \int_{\partial\Omega}  \mathrm{n}(x) \cdot \big(\int_{\mathbb{R}^3} \gamma F v\dd v \big)(s,x) \phi_1(s,x) \bd\sigma_x \bd s.
\end{split}
\end{align}
If $\phi_1(s,x)|_{\partial\Omega} = 0$ for any $ 0 \le s \le T$ and $\phi_1(0,x)=\phi_1(T,x)=0$, we can conclude: In the distribution sense
\begin{align*}
\partial_t \int_{\mathbb{R}^3} F^n \mathrm{d}v + \nabla \cdot \int_{\mathbb{R}^3} v F^n  \mathrm{d}v =0.
\end{align*}
	As for the local  conservation law of mass, similar to Lemma \ref{green-formula-Lp},  we can use the Green formula to define the trace of $\int_{\mathbb{R}^3} v F \mathrm{d} v$ on $\partial\Omega$,   Denoting it by $\gamma_x(\int_{\mathbb{R}^3} v f\mathrm{d} v)$.
 	
From \eqref{commutative},  we can infer that
\begin{align*}
\mathrm{n}(x) \cdot  \gamma_x(\int_{\mathbb{R}^3} v F \mathrm{d} v) & =
\int_{\Sigma_+^x} \gamma_+ F   | \mathrm{n}(x) \cdot v  |\mathrm{d} v  - \int_{\Sigma_-^x} \gamma_- F |\mathrm{n}(x) \cdot v| \mathrm{d} v,\\
& = \mathrm{n}(x) \cdot  \int_{\mathbb{R}^3} v \gamma f\mathrm{d} v.
\end{align*}
This means that the trace operator $\gamma$ is commutative with integral operator $\int$. This is because that the traces of solutions have the full estimate.  Besides,   the estimates of $\gamma F$ gives,
$$\mathrm{n}(x)\gamma_x(\int_{\mathbb{R}^3} F(s) v \mathrm{d} v)  \in L^1\big((0,T)\times\partial\Omega; \mathrm{d} \sigma_x \mathrm{d} s \big).$$
\subsubsection{Local  conservation law of momentum}Different with the local conservation law of mass, there exists defect measure in the conservation law of momentum. For any fixed $T>0$, multiplying the first equation of \eqref{boltzmann-approximate-lions} by $ v \phi_1$ with $\phi_1 \in \mathcal{D}((0,T)\times\Omega )$, we have after integrating by part
\begin{align}
\label{momentum-1}
\begin{split}
\int_{0}^{T} \int_\Omega \partial_t \phi_1(s) \int_{\mathbb{R}^3} v F^n(s)  \bd v \bd x \bd s +  \int_{0}^{T} \int_\Omega \nabla \phi_1(s) \int_{\mathbb{R}^3} F^n(s) v\otimes v \bd v \bd x \bd s = 0.
\end{split}
\end{align}
Recalling that \[ v F^n \rightharpoonup v F,~ \text{in} ~ L^\infty((0,T); L^1(\mathcal{O}));~~ F^n \to F,~~a.e~~\text{on}~~ {\mathcal{O}_T}, \]
by Vitalli convergence theorem, we can deduce
\[ v F^n \to v F,~ \text{in}  ~ L^\infty((0,T); L^1(\mathcal{O})). \]

Thus,
\[  \int_{0}^{T} \int_\Omega \partial_t \phi_1(s) \int_{\mathbb{R}^3} v F^n(s)  \bd v \bd x \bd s  \to  \int_{0}^{T} \int_\Omega \partial_t \phi_1(s) \int_{\mathbb{R}^3} v F(s)  \bd v \bd x \bd s.  \]
For the second term in \eqref{momentum-1},  the only thing  at our disposal is
\[ vF^n \to vF, ~~\text{in}~~ L^1((0,T)\times\mathcal{O}). \]
With these estimates, we can only prove that there exists  distribution-value matrix $\mathbf{W}$ with $\mathbf{W}_{i,j}( i, j =1, 2, 3) \in   \mathcal{D}'((0,T)\times\Omega)$  such that while $n \to \infty$
\begin{align*}
\int_{0}^{T} \int_\mathcal{O} F^n v\otimes v \cdot \nabla \phi_1 \bd v \bd x \bd s   \to  \int_{0}^{T} \int_\mathcal{O}   (F v\otimes v ) \cdot \nabla \phi_1\bd v \bd x \bd s + (\mathbf{W}, \nabla \phi_1),
\end{align*}
where $(\cdot,\cdot)$ denotes the action between distribution and test function.

Thus, we conclude the local conservation law of momentum.

\section{Proof of Theorem \ref{main-result}}
In this section, we prove Theorem \ref{main-result}. The proof   is  mainly made up of two parts. The first part consists in proving   the fluctuations $\{g_\epsilon\}$ tend to $g$ with form \eqref{limit-g}  in the interior domain. Besides, the coefficient $\uu$ and $\theta$ are weak solutions to   NSF equations. The proof of this part is the same as that in \cite{lm2010soft}.   The second part is devoted to show $\uu$ and $\theta$ satisfy    Dirichlet boundary conditions.

From now on, We  denote $\dd \varsigma = \mm |\mathrm{n}(x)\cdot v| \dd v$.  Here, we sketch the proof of the first part.  From \eqref{ns-scaling} and \eqref{be-f}, the equation of $g_\epsilon$ is
\begin{align*}
\begin{cases}
\epsilon \partial_t g_\epsilon + \vdot g_\epsilon + \frac{1}{\epsilon}\mathcal{L} g_\epsilon  =  \mathcal{B}(g_\epsilon,g_\epsilon),\\
\gamma_- g_\epsilon = z_\eps,\\
g_\epsilon(0,x,v)= \frac{1}{\epsilon}\frac{F_\epsilon^0 - \mm}{ \mm}\,,
\end{cases}
\end{align*}
where the scaled collision operator $\mathcal{B}$ is defined as
\begin{equation}\nonumber
  \mathcal{B}(G, G) = \tfrac{1}{\mathrm{M}}\Q(\mathrm{M}G, \mathrm{M}G)\,.
\end{equation}

Because of the order   $\frac{1}{\eps}$ before linear Boltzmann operator,  $g_\epsilon$ tends to an infinitesimal Maxwellian $g$ in the sense of distribution, the kernel space of linear  Boltzmann collisional operator $\mathcal{L}$,
\begin{align}
\label{g-limit}
g(t,x,v)= \rho(t,x) + \uu(t,x) \cdot v + \tfrac{\theta(t,x)}{2}(|v|^2 - 3).
\end{align}

According to the uniform relative entropy inequality\eqref{est-f-eps-entropy}, by the argument as \cite[Lemma 3.1.2]{saint2009book},  it follows that
\[ g_\eps \in L^\infty((0,+\infty),  L^1(\mathcal{O}, (1 + |v|^2) \mm \dd v \dd x), \]
and   $g_\eps$ belongs to $L^2$ space up to a $L^1$ perturbation of order $\eps$.
Moreover, $\{g_\epsilon\}$ is weakly compact in $L^\infty((0,+\infty),  L^1(\mathcal{O}, (1 + |v|^2) \mm \dd v \dd x)$. Thus,  there exist $\rho, u, \theta \in L^\infty((0,\infty); L^2(\Omega))$ such that
\begin{align}
\label{limit-g-p}
g(t,x,v) = \rho(t,x) + \uu(t,x) \cdot v + \tfrac{\theta(t,x)}{2}(|v|^2 - 3),
\end{align}
and
\[  g_\eps \weak g,~~ \text{in}~~L^\infty((0,+\infty),  L^1(\mathcal{O}, (1 + |v|^2) \mm \dd v \dd x).\]
Now, we turn to derive   equations of $\rho$, $u$ and $\theta$. Let
\begin{align*}
\beta(z) =  \tfrac{z-1}{ 1 + (z-1)^2}, ~~z>0.
\end{align*}
Denoting $\tilde{g}_\eps = \frac{1}{\eps} \beta(G_\eps )$ and $N_\eps = 1 + \eps^2 g_\eps^2$, then by simple computation
\[  \tilde{g}_\eps = \frac{g_\eps}{ N_\eps},~~ \eps g_\eps>-1.  \]
By the relative entropy estimates, for any $T>0$,
\begin{align}
\label{est-g-eps-non}
\tilde{g}_\eps \in L^\infty((0,T); L^2(\mm \dd v \dd x )).
\end{align}
By simple computation,
\[ \tilde{g}_\eps \weak g, ~~ \text{in}~~ L^\infty((0,T); L^2(\mm \dd v \dd x )).  \]
Denoting
\[\la f \ra := \int_{\mathbb{R}^3} f \mm \dd v,  \]
  the equations of $\rho$, $\uu$ and $\theta$ can be derived from equations of $\la \tilde{g}_\eps \ra$, $ \la v \tilde{g}_\eps \ra$ and $\la (\frac{|v|^2}{3} - 1 ) \tilde{g}_\epsilon \ra$. Indeed, by the existence Theorem \ref{main-result-cut}, for every $\eps >0$, the initial boundary problem (\ref{be-f}) admits  renormalized solutions $g_\eps$ such that
\begin{align}
\label{renormalize-non}
\partial_t \tilde{g}_\eps +   \tfrac{1}{\eps}v \cdot \nabla_x \tilde{g}_\eps = \tfrac{1}{\eps^3   } (\tfrac{2}{N_\eps^2} - \tfrac{1}{N_\eps})\B(G_\eps,G_\eps)
\end{align}
holds in the sense of distribution.

Multiplying the above equation by $1,v$ and $|v|^2$ respectively, we get the approximate conservation laws of moments:
\begin{align*}
\partial_t \la \tilde{g}_\eps\ra+ \tfrac{1}{\epsilon} {\rm div}_x \la v \tilde{g}_\eps \ra = \tilde{\mathrm{D}}(1),\\
\partial_t \la v \tilde{g}_\eps \ra + \tfrac{1}{\epsilon} {\rm div}_x \la \mathrm{A}(v)\tilde{g}_\eps \ra  +  \tfrac{1}{3\epsilon} \nabla_x \la |v|^2\tilde{g}_\eps \ra = \tilde{\mathrm{D}}(v ),\\
\partial_t \la (\tfrac{|v|^2}{3} - 1) \tilde{g}_\eps \ra + \tfrac{2}{3\epsilon} {\rm div}_x \la \mathrm{B}(v)\tilde{g}_\eps \ra  +  \tfrac{2}{3\epsilon} \nabla_x \la v\tilde{g}_\eps \ra = \tilde{\mathrm{D}}((\tfrac{|v|^2}{3} - 1) ),
\end{align*}
with
\begin{align*}
 \tilde{\mathrm{D}}(\xi ) = \tfrac{1}{\epsilon^3} \la   { \xi} (\tfrac{2}{N_\eps^2} - \tfrac{1}{N_\eps})\B(G_\eps,G_\eps) \ra.
\end{align*}
Moreover, from \cite[Sec. 6]{lm2010soft}, for every subsequence $\{g_{\eps_k}\}$ converging to $g$ in $L^\infty((0,+\infty),  L^1(\mathcal{O}, (1 + |v|^2) \mm \dd v \dd x)$   as $\eps_k \to 0$, the subsequence also enjoys
\begin{align*}
\mathbf{P}\la v g_{\epsilon_k} \ra \to \uu,~~\la (\tfrac{|v|^2}{5} - 1) \tilde{g}_\eps   \ra \to \theta,~~ \text{in} ~~ C([0,\infty)\,,  \mathcal{D}'(\mathbb{R}^3)),
\end{align*}
and in the distributional sense,
\[ \mathbf{P}(\tfrac{1}{\epsilon} {\rm div}_x \la \mathrm{A}(v) \tilde{g}_\eps \ra) \to \uu \cdot \nabla \uu -  \nu \Delta \uu,  \]
\[ \tfrac{2}{5\epsilon} {\rm div}_x \la \mathrm{B}(v)\tilde{g}_\eps \ra \to \uu\cdot \nabla \theta - \kappa \Delta \theta.  \]
Moreover, $\rho,  \uu, \theta$ in \eqref{limit-g-p} satisfy
\[ {\rm div} \uu = 0, ~ \rho + \theta = 0.\]
With the notation $q_\eps$ for $ \frac{1}{\eps^2   } (\frac{2}{N_\eps^2} - \frac{1}{N_\eps})\B(G_\eps,G_\eps)$,
\begin{align}
\label{limit-convergence-q}
q_\eps \rightharpoonup q =   v \cdot \nabla_x g, ~~\text{in} ~~L^1(((0,T)\times\mathcal{O}); \mm \dd v \dd x \dd t).
\end{align}

The energy estimate of $u$ and $\theta$ \eqref{est-u-theta-f} can be obtained by combining the relative entropy inequality and \eqref{limit-convergence-q}, see \cite[Sec. 6]{lm2010soft}.

{\bf Verify Dirichlet boundary condition.}

In the interior domain,  the fluctuations $\{g_\epsilon\}$ tend to an infinitesimal Maxwellian as $\epsilon $ tends to zero. At the mean time, NSF equations can be derived.  In what follows, we try to derive the boundary conditions of $u$ and $\theta$. First, we prepare some estimates on the traces of $\tilde{g}_\eps$.
\begin{lemma}
	\label{lemma-g-hat}
	Under the same assumptions of Theorem \ref{main-result}, there exists  a constant $C$ such that
	\begin{align}
	\label{est-g-gamma-boundary}\|\gamma (\tilde{g}_\epsilon) (t) \|_{L^2(\dd \varsigma  \dd \sigma_x)}^2\le C \cdot  C_0 \epsilon.
	\end{align}
\end{lemma}
\begin{remark}\label{remark-trace}
	From \cite[pp.1273]{masmoudi-srm2003stokesfourier}, for the solutions with Maxwell reflection boundary, they could only infer  some estimates like $(\gamma_+ g_\epsilon  - \la \gamma_+ g_\epsilon\ra|_{\partial\Omega})$ from the relative entropy estimates. But for incoming boundary condition, we can directly deduce some full estimate of  $\gamma g_\epsilon$ from the relative entropy estimate \eqref{est-f-eps-entropy}.
\end{remark}

\begin{proof}
\begin{align*}
\tilde{g}_\eps^2 & = \big( \frac{g_\eps}{1 + (\eps g_\eps)^2} \big)^2  =  \frac{g_\eps^2}{1 + \tfrac{1}{3} \eps g_\eps } \cdot \frac{1 + \frac{1}{3} \eps g_\eps}{(1 + (\eps g_\eps)^2)^2}   \le C \cdot \frac{g_\eps^2}{1 + \tfrac{1}{3} \eps g_\eps }.
\end{align*}
where $C = \sup_{z \ge -1} \frac{1 + \frac{z}{3}}{( 1 + {z^2})^2} $.

On the other hand, from  \cite[Lemma 8.1]{diogo2012},
\begin{align*}
\frac{g_\eps^2}{1 + \tfrac{1}{3} \eps g_\eps } \le 2 \frac{h(G_\eps)}{\eps^2}.
\end{align*}
Recalling that
\[ \frac{1}{\epsilon^2}\int_{0}^{t}\int_{\Sigma_\pm }h(\gamma  G_\epsilon  ) \mathrm{d}\varsigma \mathrm{d}\sigma_x\mathrm{d}s \le C_0 \eps, \]
we complete the proof.
\end{proof}

With the estimates \eqref{est-u-theta-f} and \eqref{est-g-gamma-boundary}  at our disposal, we can get the boundary conditions of $u$ and $\theta$.
\begin{lemma}
	\label{lemma-boundary}
	Under the same assumptions of Theorem \ref{main-result}, the traces of the limiting fluctuation
	\[g|_{\partial\Omega} \in L^1_{loc}(\dd \sigma_x \dd t; L^1(\dd \varsigma))  \]
	satisfy the identity
	\[  g|_{\partial\Omega} = \uu|_{\partial\Omega} \cdot v + \theta|_{\partial\Omega}\frac{|v|^2-5}{2}.\]
	Furthermore,
\begin{align*}
\uu|_{\partial\Omega} =  \theta|_{\partial\Omega} = 0.
\end{align*}
\end{lemma}
\begin{proof} For any $T \ge 0$,  multiplying the following equation by $\psi \mm$ with $\psi \in \mathcal{D}([0,T]\times \bar{\Omega}\times\mathbb{R}^3)$
	\begin{align*}
	\eps \partial_t \tilde{g}_\eps +    v \cdot \nabla_x \tilde{g}_\eps = \tfrac{1}{\eps^2   } (\tfrac{2}{N_\eps^2} - \tfrac{1}{N_\eps})\B(G_\eps,G_\eps) = {q}_\eps,
	\end{align*}
	it follows that
	\begin{align*}
	0  = & -   \int_{0}^{T} \int_{\mathcal{O}}    {q}_\eps  \psi \mm     \mathrm{d}v \dd x\mathrm{d}t + \epsilon \int_{\mathcal{O}} \tilde{g}_\epsilon(T) \psi \mm   \mathrm{d}v \dd x \\
	& - \epsilon \int_{\mathcal{O}} \tilde{g}_\epsilon^0 \psi \mm   \mathrm{d}v \dd x +   \int_{0}^{T}  \int_{\Sigma_\pm} \gamma_\pm \tilde{g}_\eps \psi  \mathrm{n}(x) \cdot v \mm \dd v \dd \sigma_x\dd s \\
	&  -   \int_{0}^{T} \int_{\mathcal{O}} \tilde{g}_\eps(\epsilon\partial_t \psi  + v \cdot \nabla_x \psi)  \mm    \mathrm{d}v \dd x\mathrm{d}t.
	\end{align*}
	As the mean free path goes to zero, with the help of \eqref{limit-convergence-q} and Lemma \ref{lemma-g-hat},
	\begin{align*}
	\int_0^T\int_\Omega\la g \cdot   ( v \cdot \nabla_x )\psi \ra \dd x \dd s +  \int_0^T\int_\Omega\la \psi \cdot   ( v \cdot \nabla_x ) g \ra \dd x \dd s = 0.
	\end{align*}
	Since $\psi$ can be chosen arbitrarily in $\mathcal{D}([0,T]\times \bar{\Omega}\times\mathbb{R}^3)$.  The above equation gives
	\[ \gamma(g) = \uu|_{\partial\Omega} \cdot v + \theta|_{\partial\Omega}\tfrac{|v|^2-5}{2} = 0.  \]
	Thus,
	\begin{align*}
	\uu|_{\partial\Omega} = \theta|_{\partial\Omega} = 0.
	\end{align*}
\end{proof}

\end{document}